\documentclass[reqno]{amsart}
\usepackage[shortlabels]{enumitem}
\usepackage{graphicx}
\usepackage{epstopdf}
\usepackage{caption}
\usepackage{subcaption}
\usepackage{color}
\usepackage{pdfpages}
\usepackage{marginnote}
\usepackage{amssymb}
\usepackage{mathrsfs}
\usepackage{enumitem}
\usepackage{hyperref}
\usepackage{isomath}
\usepackage{bbm}
\usepackage[all]{xy}

\begin{document}

\newtheorem{thm}{Theorem}[section]
\newtheorem*{thm1}{Theorem}
\newtheorem{lem}[thm]{Lemma}
\newtheorem*{lem1}{Lemma}
\newtheorem{cor}[thm]{Corollary}
\newtheorem{add}[thm]{Addendum}
\newtheorem{prop}[thm]{Proposition}
\theoremstyle{definition}
\newtheorem{defn}[thm]{Definition}
\newtheorem{claim}[thm]{Claim}
\newtheorem*{mainthm}{Theorem 1.1}
\newtheorem*{maincor}{Corollary 1.2}
\newtheorem*{prop1}{Proposition 2.2'}
\theoremstyle{remark}
\newtheorem{rmk}[thm]{Remark}
\newtheorem{ex}[thm]{Example}
\newtheorem{conj}[thm]{Conjecture}
\newtheorem{ob}[thm]{Observation}
\newcommand{\CC}{\mathbb{C}}
\newcommand{\RR}{\mathbb{R}}
\newcommand{\DD}{\mathbb{D}}

\newcommand{\ZZ}{\mathbb{Z}}
\newcommand{\QQ}{\mathbb{Q}}
\newcommand{\NN}{\mathbb{N}}
\newcommand{\FF}{\mathbb{F}}
\newcommand{\PP}{\mathbb{P}}
\newcommand{\CmodTwoPiIZ}{{\mathbf C}/2\pi i {\mathbf Z}}
\newcommand{\Cnozero}{{\mathbf C}\backslash \{0\}}
\newcommand{\Cinfty}{{\mathbf C}_{\infty}}
\newcommand{\RPnminustwo}{\mathbb{RP}^{n-2}}

\newcommand{\SLtwoC}{\mathrm{SL}(2,{\mathbb C})}
\newcommand{\GLtwoC}{\mathrm{GL}(2,{\mathbb C})}
\newcommand{\SLtwoR}{\mathrm{SL}(2,{\mathbb R})}
\newcommand{\PSLtwoC}{\mathrm{PSL}(2,{\mathbb C})}
\newcommand{\PSLtwoR}{\mathrm{PSL}(2,{\mathbb R})}
\newcommand{\SLtwoZ}{\mathrm{SL}(2,{\mathbb Z})}
\newcommand{\PSLtwoZ}{\mathrm{PSL}(2,{\mathbb Z})}

\newcommand{\A}{{\mathcal A}}
\newcommand{\B}{{\mathcal B}}
\newcommand{\C}{{\mathcal X}}
\newcommand{\D}{{\mathcal D}}
\newcommand{\E}{{\mathcal E}}
\newcommand{\F}{{\mathcal F}}
\newcommand{\T}{{\mathcal T}}

\newcommand{\MCG}{\mathcal{MCG}}
\newcommand{\EE}{\mathbb{E}^2}
\newcommand{\HH}{\mathbb{H}^2}
\newcommand{\HHH}{\mathbb{H}^3}
\newcommand{\tr}{{\hbox{tr}\,}}
\newcommand{\Hom}{\mathrm{Hom}}
\newcommand{\Aut}{\mathrm{Aut}}
\newcommand{\Inn}{\mathrm{Inn}}
\newcommand{\Out}{\mathrm{Out}}
\newcommand{\SL}{\mathrm{SL}}
\newcommand{\BQ}{\rm{BQ}}
\newcommand{\Id}{\rm{Id}}

\newcommand{\setn}{{[n]}}
\newcommand{\powern}{{P(n)}}
\newcommand{\nck}{{C(n,k)}}

\newcommand{\hatI}{{\hat{I}}}
\newcommand{\TkDelta}{{T^{|k|}(\Delta)}}
\newcommand{\vecDelta}{{\vec{\Delta}_{\phi}}}

\newcommand{\Tabstwo}{{T^{|2|}(\Delta)}}
\newcommand{\Tnminusone}{{T^{|n-1|}(\Delta)}}
\newcommand{\Hur}{{\mathcal{H}}}

\newenvironment{pf}{\noindent {\it Proof.}\quad}{\square \vskip 10pt}

\title[Carrier graphs for representations of the rank two free group]{Carrier graphs for representations of the rank two free group into isometries of hyperbolic three space}
\author[S.P. Tan and B. Xu]{Ser Peow Tan and Binbin Xu}
\address{Department of Mathematics \\
National University of Singapore \\
Singapore 119076} \email{mattansp@nus.edu.sg}
\address{Departement des Math\'eques  \\
Universit\'{e} du Luxembourg, Luxembourg} \email{binbin.xu@uni.lu}

\subjclass[2000]{}

%
%

 \begin{abstract}
 Carrier graphs were first introduced for closed hyperbolic 3-manifolds by White. In this paper, we first generalize this definition to carrier graphs  for representations of a rank two free group into the isometry group of hyperbolic three space. Then we prove the existence and the finiteness of minimal carrier graphs for those representations which are discrete, faithful and geometrically finite, and more generally, those that satisfy certain finiteness conditions first introduced by Bowditch.
 \end{abstract}

 \maketitle
 \tableofcontents

 \vspace{10pt}
\section{Introduction}
Let $M$ be a hyperbolic $3$-manifold of rank $n$. A connected finite graph $\Gamma$ is called an {\textit{$n$-graph}} if all vertices are trivalent and the rank of the graph is $n$. A {\textit{carrier $n$-graph}} for $M$ is an $n$-graph together with a map $f:\Gamma\rightarrow M$ such that $f_\ast:\pi_1(\Gamma)\rightarrow\pi_1(M)$ is an epimorphism. This definition was first introduced by White in \cite{white2002cag} for a closed connected hyperbolic 3-manifold. A carrier $n$-graph is said to be {\textit{minimal}} if the image of $f$ has the minimal length among all carrier $n$-graphs for $M$. By studying the minimal carrier graph, White showed that the injectivity radius of a closed hyperbolic 3-manifold $M$ is bounded above by a constant only depending on its rank. In \cite{siler2012gd}, Siler studied the geometric properties of the carrier graph for a general hyperbolic $3$-manifold. However the existence and uniqueness of the minimal carrier graph of a hyperbolic $3$-manifold in general is still unclear. In this paper, we consider this problem for geometrically finite $M$ whose fundamental group is isomorphic to $\FF_2$ the rank two free group.

Any such $M$ is isometric to a quotient of $\HHH$ by the action of the image of a representation $\rho:\FF_2\rightarrow\PSLtwoC$ which is discrete and faithful. By the result of Culler in \cite{culler1986aim}, these representations can be lifted to representations in $\SLtwoC$ up to $\pm \Id$. Hence any such $M$ can be considered as a point in the $\SLtwoC$-character variety  $\chi(\SLtwoC)$ which consists of all representations of $\FF_2$ to $\SLtwoC$ up to conjugacy. There is an open subset in $\chi(\SLtwoC)$ which is invariant under the action of $Out(\FF_2)$ satisfying certain conditions first introduced by Bowditch. Moreover, the $Out(\FF_2)$-action on it is properly discontinuous. We extend our study to the representations corresponding to the points in this open subset.

The character variety $\chi=\chi(\SLtwoC)$ has been studied by many people. The theorem of Vogt and Fricke implies that it can be identified with $\CC^3$, where the coordinates are given by taking the traces of the image of a fixed superbasis of $\FF_2$. By considering the natural embedding of $\SLtwoR$ into $\SLtwoC$, this space also contains Fricke spaces of orientable hyperbolic surfaces with Euler characteristic $-1$. In \cite{goldman2009ems}, Goldman described those Fricke spaces as subspaces of $\chi$ using the above coordinate system. Of particular interest is the Fricke space of a once-puncture torus, where in \cite{mcshane1991thesis}, McShane proved a remarkable identity. The coordinates for this space satisfies the Markoff equation. By using this fact, in \cite{bowditch1996blms}, Bowditch gave an alternative proof of this identity, and then in \cite{bowditch1998plms}, generalized this identity to the type preserving representations satisfying certain conditions which he called {\textit{Q-conditions}}. In \cite{tan-wong-zhang2008advm}, by dropping the type-preserving restriction, Tan, Wong and Zhang proved a variation of McShane identity for general representations satisfying the {\textit{Bowditch Q-conditions}}, or {\textit{BQ-conditions}} for short. A representation satisfying BQ-conditions will be called a {\textit{BQ-representation}}.

In this paper, we will consider the subspace of conjugacy classes of all BQ-representations and we call it the {\textit{Bowditch space}}. The Bowditch space contains the Schottky representations, that is, the discrete faithful convex cocompact representations arising from the holonomy representations of convex cocompact hyperbolic three manifolds with fundamental group $\FF_2$. However, the whole subspace is strictly larger than that, see \cite{tan-wong-zhang2008advm} or \cite{STY}. For example, it contains representations of $\FF_2$ into $\SLtwoR$ arising from hyperbolic structures on the torus with one cone singularity. When the cone angle is irrational, such a representation is not discrete. 

We define the carrier $2$-graph for a generic irreducible representation $\rho$ which is a generalization of the one defined by White in \cite{white2002cag}. Roughly speaking, a carrier $2$-graph for $\rho$ is an equivalence class of a pair $(\Gamma,\Psi)$, where $\Gamma$ is a $2$-graph marked by a basis of $\FF_2$, $\Psi$ is a $\rho$-equivariant homeomorphism from the universal cover $\widetilde{\Gamma}$ to $\HHH$ and the equivalence relation is defined by considering the $\Inn(\FF_2)$-action on the pair. In particular, $\Psi$ induces an automorphism $h_\Psi$ of $\FF_2$. By pulling back the intrinsic metric on $\Psi(\widetilde{\Gamma})$ induced by the hyperbolic metric, we obtain a metric on $\Gamma$, and the {\textit{length}} of a carrier $2$-graph is defined to be the sum of its edge lengths with respect to this metric. A carrier $2$-graph is {\textit{minimal}} if it has the shortest length among all carrier $2$-graphs.

To study the existence of the minimal carrier $2$-graph, we decompose the space of all carrier $2$-graph into a disjoint union of subsets by considering $h_\Psi$. In each subset, a carrier $2$-graph with the shortest length, if it exists, is called a  {\textit{critical carrier $2$-graph}}. Then our result can be stated as follows:
\begin{thm}\label{main}
	If $\rho:\FF_2\rightarrow\SLtwoC$ satisfies the BQ-conditions, then  critical carrier $2$-graphs exist and there are finitely many of them.
\end{thm}
Since a minimal carrier $2$-graph is also a critical carrier $2$-graph, as a corollary, we show that
\begin{cor}
	If a representation $\rho$ satisfies the BQ-conditions, then it admits finitely many minimal carrier $2$-graphs.
\end{cor}
The main ingredients in the proof is the convexity lemma in hyperbolic geometry and a detailed analysis of the Steiner tree for a triple of pairwise disjoint geodesics in $\HHH$. Informally, the Steiner tree is a graph connecting triple of pairwise disjoint geodesics with minimal length among all such graphs. This is a generalization of the Steiner tree for points in Euclidean plane.

The outline of the paper is as follows. In section $2$, we recall some necessary background on $3$-dimensional hyperbolic geometry.

In section $3$, we recall background on rank two free groups and the definition of the Bowditch Q-conditions. 

In section $4$, we will recall the definition of the tree of superbases for $\FF_2$ and discuss three different edge orientations on it induced by an irreducible representation of $\FF_2$ into $\SLtwoC$. In particular, we show that if the representation satisfies the Bowditch Q-conditions, then for any edge orientation introduced in this section, there is a compact attracting subtree.

In section $5$, we prove the existence and uniqueness of the Fermat point of a hyperbolic triangle, which is a result similar to the Euclidean case.

In section $6$, we generalize the definition of the Fermat point result for a triple of geodesics in $\HHH$, from which we define the Steiner tree and prove its existence. We also give a characterization of its combinatorial type.

In section $7$, we define the carrier graph for a representation.

In section $8$, we give the proof of the main theorem and its corollary. 

In the last section, we will discuss the result for BQ-representations which preserve a hyperbolic plane in $\HHH$ as examples.

 \bigskip

 \noindent {\it Acknowledgements}. We are grateful to Ara Basmajian, Martin Bridgeman, Jaejeong Lee, Hugo Parlier, Jean-Marc Schlenker and Andrew Yarmola for helpful conversations and comments. We would like to thank Jonah Gaster for bringing our attention to the work of White \cite{white2002cag} and Silver \cite{siler2012gd}. The work has been done during the second author's several visits to the mathematics department of National University of Singapore. He would like to thank them for their hospitality. These visits were partially supported by the Post-Doc fellowship from Korea Institute for Advanced Study. Tan was partially supported by the National University 	of Singapore academic research grant R-146-000-235-114.

 \bigskip

\section{Three dimensional hyperbolic geometry}
\subsection{Upper-half space model}
The upper half space model of $3$-dimensional hyperbolic space is defined to be the following set:
\begin{displaymath}
\HHH=\{(a,b,c)\in\RR^3\mid c>0\},
\end{displaymath}
equipped with the hyperbolic metric:
\begin{displaymath}
(\mathrm{d}s)^2=\frac{(\mathrm{d}a)^2+(\mathrm{d}b)^2+(\mathrm{d}c)^2}{c^2}.
\end{displaymath}
The boundary $\partial\HHH$ of $\HHH$ is the one point compactification of the plane: $$\{(a,b,c)\in\RR^3\mid c=0\}.$$
It can be identified with the Riemann sphere $\hat{\CC}=\CC\cup\infty$ with $u=a+ib$.
 
In the upper half space model, the geodesics are either vertical lines or half circles orthogonal, at its both ends, to the plane defined by $c=0$. Each geodesic is uniquely determined by its end points on $\partial\HHH$, hence there is a natural 1-1 correspondence between (oriented) geodesics and (ordered) pairs of distinct points on $\partial\HHH$. Given two distinct points $u$ and $u'$ in $\hat{\CC}$, we denote by $[u,u']$ the non-oriented geodesic determined by them and by $(u,u')$ (resp. $(u',u)$) its two oriented versions oriented from $u$ to $u'$ (resp. from $u'$ to $u$).

Given two distinct geodesics, their positions relative to each other have three possible types:
\begin{itemize}
	\item intersecting in $\HHH$,
	\item parallel, i.e. sharing one end point in $\partial\HHH$,
	\item disjoint (or ultra-parallel), i.e. disjoint in $\HHH\cup\partial\HHH$.
\end{itemize}
The distance between two disjoint geodesics is positive and can be realized by their intersection points with their common perpendicular geodesic. The distance between two parallel geodesics is $0$ which is not realizable. The distance between two intersecting geodesics is $0$ realized by their intersection point. 

\subsection{Orientation preserving isometries of \texorpdfstring{$\HHH$}{Lg}}

The orientation preserving isometry group of $\HHH$ can be identified with $\PSLtwoC$. Its elements can be classified by their fixed points in $\HHH\cup\partial\HHH$. A non-identity element in $\PSLtwoC$ is called
\begin{itemize}
	\item {\textit{loxodromic}} if it preserves a geodesic and acts as translation on it;
	\item {\textit{parabolic}} if it has a unique fixed point contained in $\partial\HHH$;
	\item {\textit{elliptic}} if it fixes a geodesic pointwise in $\HHH$.
\end{itemize}
In particular, we call an element in $\PSLtwoC$ an {\textit{involution}} if it is elliptic of order $2$. The geodesic fixed by a loxodromic element or an elliptic element is called its {\textit{axis}}. 

\begin{defn}
	The {\textit{real translation distance}} of $\phi\in\PSLtwoC$ is defined to be the infimum:
	\begin{equation*}
	a(\phi):=\inf_{p\in\HHH}\mathrm{d}_{\HHH}(p,\phi(p)),
	\end{equation*}
	where $\mathrm{d}_{\HHH}$ is the distance in $\HHH$ with respect to the hyperbolic metric.
\end{defn}

If $\phi$ is a loxodromic element, then $a(\phi)$ is positive and realized by any point on the axis of $\phi$; if $\phi$ is elliptic, then $a(\phi)$ is $0$ and realized by any point on the axis of $\phi$; if $\phi$ is parabolic, then $a(\phi)$ equals $0$ and is not realizable. 

A well known fact is that any element $\phi\in\PSLtwoC$ can be written as a composition of two involutions of $\HHH$. Consider the axes of these two involutions. For any $\phi$ different from identity, they are
\begin{itemize}
	\item disjoint, if $\phi$ is loxodromic, and the distance between them equals to $a(\phi)/2$;
	\item parallel, if $\phi$ is parabolic;
	\item intersecting with each other, if $\phi$ is elliptic.
\end{itemize}
This decomposition is not unique. Now consider two elements $\phi$ and $\psi$ in $\PSLtwoC$. We have the following well-known theorem:
\begin{thm}[Coxeter extension]\label{irre}
	Two elements $\phi$ and $\psi$ do not have any common fixed point on $\partial\HHH$, if and only if there is a unique triple of involutions $(\iota_1,\iota_2,\iota_3)$ such that 
	\begin{eqnarray}\label{rotation}
	\phi=\iota_3\iota_2,\nonumber\\
	\psi=\iota_2\iota_1,\\
	(\phi\psi)^{-1}=\iota_1\iota_3.\nonumber
	\end{eqnarray}
\end{thm}
\begin{rmk}
	In the next section, we will see that the condition that $\phi$ and $\psi$ do not have common fixed points in $\partial\HHH$ is equivalent to the condition that the subgroup of $\PSLtwoC$ generated by $\phi$ and $\psi$ is irreducible.
\end{rmk}

Assume that $\phi$ and $\psi$ do not have common fixed points on $\partial\HHH$. We consider $(\iota_1,\iota_2,\iota_3)$ the unique triple of involutions associated to $(\phi,\psi,(\phi\psi)^{-1})$. Let $\gamma_1$, $\gamma_2$ and $\gamma_3$ be the the axes of $\iota_1$, $\iota_2$ and $\iota_3$ respectively. If moreover $\phi$, $\psi$ and $(\phi\psi)^{-1}$ are all loxodromic elements, then the $\gamma_i$'s are disjoint in $\HHH\cup\partial\HHH$. Therefore we have a unique right angled hexagon which is bounded by $\gamma_i$'s and their pairwise common perpendicular geodesics. It is possible that this right angled hexagon is degenerate, since the sides along $\gamma_i$'s may have zero length.

\subsection{Lifting \texorpdfstring{$\PSLtwoC$}{Lg} to \texorpdfstring{$\SLtwoC$}{Lg}}
In order to talk about the trace, we would like to consider $\SLtwoC$ instead of $\PSLtwoC$. A reference for this part can be found in Section V and Section VI in \cite{fenchel}.

Any $\phi\in\PSLtwoC$ has two lifts $\xi$ and $-\xi$ in $\SLtwoC$. The actions of $\pm\xi$ on $\HHH$ are the same as that of $\phi$. In particular, the subgroup $\pm \Id$ acts trivially on $\HHH$. Therefore, the action of $\SLtwoC$ on $\HHH$ is not free. If $\phi$ is a loxodromic (resp. parabolic, elliptic) element, then its lifts are also called loxodromic (resp. parabolic, elliptic) elements in $\SLtwoC$. Elements of different types in $\SLtwoC$ can be distinguished by considering their traces: an element $\xi\in\SLtwoC$ different from $\pm\Id$ is
\begin{itemize}
	\item loxodromic if and only if its trace is not contained in the segment $[-2,2]$;
	\item parabolic if and only if its trace is $2$ or $-2$;
	\item elliptic if and only if its trace is contained in $(-2,2)$.
\end{itemize}
As with $\PSLtwoC$-elements, the geodesic fixed by a loxodromic element or an elliptic element is called its axis. The translation distance $a(\xi)$ of $\xi\in\SLtwoC$ can be defined in the same way by considering their action on $\HHH$.

From the discussion in the previous part, any isometry $\phi$ can be written as a composition $\iota_2\iota_1$ of two involutions. Hence an alternative way to find the lift of $\phi$ is by considering the lifts of $\iota_2$ and $\iota_1$. 

More precisely, let $\iota$ be an involution whose axis is $[u,u']$, and consider its two associated oriented geodesics $(u,u')$ and $(u',u)$. The two lifts of $\iota$ associated to $(u,u')$ and $(u',u)$ respectively are:
\begin{displaymath}
\displaystyle{\frac{i}{u-u'}}\left(\begin{array}{cc}
u+u' & -2uu'\\
2 & -u-u'		
\end{array}\right)\quad \textrm{and}\quad
\displaystyle{\frac{i}{u'-u}}\left(\begin{array}{cc}
u+u' & -2uu'\\
2 & -u-u'		
\end{array}\right).
\end{displaymath}
If $u=\infty$ , then they become 
\begin{displaymath}
\left(\begin{array}{cc}
i & -2u'i\\
0 & -i		
\end{array}\right)\quad \textrm{and}\quad
\left(\begin{array}{cc}
-i & 2u'i\\
0 & i		
\end{array}\right).
\end{displaymath}
We denote $(u,u')$ and $(u',u)$ by $\gamma$ and $\overline{\gamma}$, and the corresponding lifts by $r$ and $\overline{r}$ respectively. We call them the {\textit{$\pi$-rotations}}. 
Moreover we can verify that:
\begin{eqnarray*}
	&&r^2=\overline{r}^2=-\mathrm{Id},\\
	&&r\overline{r}=\overline{r}r=\mathrm{Id},
\end{eqnarray*}
which then implies that
\begin{equation*}
	r^{-1}=-r=\overline{r}.
\end{equation*}

Each involution $\iota_i$ has two lifts $r_i$ and $\overline{r}_i$. Then there are four different combinations satisfying the following relations:
\begin{eqnarray*}
	\xi=r_2r_1=\overline{r}_2\overline{r}_1,\\
	-\xi=\overline{r}_2r_1=r_2\overline{r}_1,
\end{eqnarray*}  
where $\xi$ and $-\xi$ are the two lifts of $\phi$. 

Therefore, the $\SLtwoC$ version of Theorem \ref{irre} can be stated as follows:
\begin{prop1}\label{irre1}
	Let $\xi$ and $\eta$ be two elements of $\SLtwoC$. Then $\xi$ and $\eta$ do not have common fixed points on $\partial\HHH$, if and only if there exists exactly two triples of $\pi$-rotations: $(r_1,r_2,r_3)$ and its inverse $(\overline{r}_1,\overline{r}_2,\overline{r}_3)$, such that
		\begin{eqnarray*}
			\xi&=&r_3 r_2=\overline{r}_3\overline{r}_2,\\
			\eta&=&r_2r_1=\overline{r}_2\overline{r}_1,\\
			(\xi\eta)^{-1}&=&-r_1r_3=-\overline{r}_1\overline{r}_3.
		\end{eqnarray*}
\end{prop1}

A {\textit{double-cross}} $(\gamma; \gamma_1,\gamma_2)$ consists of a triple of distinct oriented geodesics, such that $\gamma$ is the common perpendicular geodesic of $\gamma_1$ and $\gamma_2$. We denote the end points of $\gamma$, $\gamma_1$ and $\gamma_2$, by $(u,u')$, $(v,v')$ and $(w,w')$ respectively. We denote by $p_1$ and $p_2$ the intersection points of $\gamma_1$ and  $\gamma_2$ with $\gamma$ respectively. Let $P_1$ (resp. $P_2$) denote the hyperbolic plane containing $\gamma$ and $\gamma_1$ (resp. $\gamma_2$). We denote by $P_1^+$ (resp. $P_2^+$) the half plane in $P_1$ (resp. $P_2$) bounded by $\gamma$ containing $p_1v'$ (resp. $p_2w'$). 

The {\textit{angle}} of the double-cross $(\gamma; \gamma_1,\gamma_2)$ is defined to be the  angle from $P_1^+$ to $P_2^+$ with respect to the orientation of $\gamma$, taking values in $[0,2\pi)$. To each double-cross, we can associate to it a complex number $l(\gamma; \gamma_1,\gamma_2)$, such that the real part is the signed distance from $p_1$ to $p_2$, and the imaginary part is the angle of the double-cross.

Given four distinct points $u_1$, $u_2$, $u_3$ and $u_4$ in Riemann sphere, we can define their cross-ratio:
\begin{equation*}
	\mathrm{Cr}(u_1,u_2;u_3,u_4):=\frac{(u_1-u_3)(u_2-u_4)}{(u_1-u_4)(u_2-u_3)}.
\end{equation*}
Then we have the following relation:
\begin{equation*}
	\mathrm{Cr}(v',w'; u',u)=\exp(l(\gamma; \gamma_1,\gamma_2))
\end{equation*}

\begin{center}
	\includegraphics[scale=0.7]{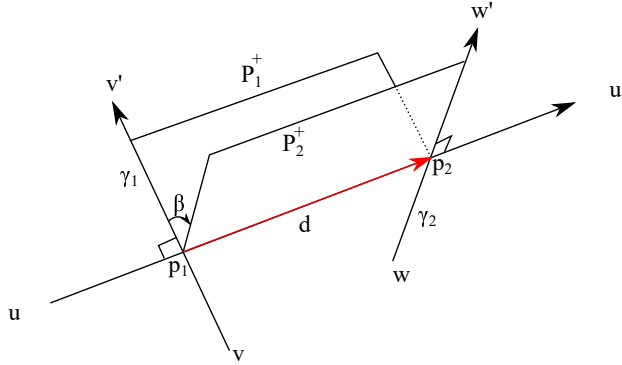}
	\captionof{figure}{$l(\gamma;\gamma_1,\gamma_2)=d+i\beta$}
\end{center}

\medskip

Let $\xi\in\SLtwoC$ be a non parabolic element different from the identity. To a decomposition $\xi=r_2r_1$, we can associate a double-cross $(\gamma; \gamma_1,\gamma_2)$ where $\gamma=\gamma(\xi)$ is the oriented axis of $\xi$, and $\gamma_i$ is the oriented axis of $r_i$. We can verify the following relation:
\begin{equation}\label{tracerotation}
\mathrm{tr}(\xi)=\mathrm{tr}(r_2r_1)=-2\cosh l(\xi).
\end{equation}
We call twice of $l(\gamma(\xi);\gamma_1,\gamma_2)$ the \textit{complex translation distance} of $\xi$. Notice that this quantity is independent of choice of the decomposition of $\xi$, therefore we will simply denote it by $l(\xi)$.

\begin{rmk}
	If $\xi$ is parabolic with $\xi=r_2r_1$, then their axes $\gamma_1$ and $\gamma_2$ share a endpoint $u$ in the boundary of $\HHH$ which is the fixed point of the action of $\xi$. We may make the following convention: if two oriented geodesic $\gamma_1$ and $\gamma_2$ point to or away from $u$ at the same time, then the complex translation distance of $\xi$ is $0$; otherwise the complex translation distance of $\xi$ is $2\pi i$. The the above formula (\ref{tracerotation}) can be extended to parabolic elements.
\end{rmk}

\section{Free group of two generators and its representations into \texorpdfstring{$\SLtwoC$}{Lg}}
\subsection{Superbases and the corresponding tree}
Let us consider the rank two free group $\FF_2$. In this part, we recall superbases for $\FF_2$ which can be viewed as vertices of a natural trivalent tree upon which the outer-automorphism group of $\FF_2$ acts. The notion of superbases was first introduced by Conway \cite{conway} for $\ZZ^2$. Superbases and the corresponding trivalent tree play an important role in the work of Bowditch \cite{bowditch1998plms} and its generalization by Tan-Wong-Zhang \cite{tan-wong-zhang2008advm}, as well as in ours. For more details about these objects, the reader may refer to \cite{goldman2009ems}, \cite{charettedrumgoldman2016} and \cite{goldmanmcshanestantchevtan2016}.

\medskip

A {\textit{basis}} of $\FF_2$ is an ordered pair $(X,Y)\in\FF_2\times\FF_2$, such that $X$ and $Y$ generate $\FF_2$ freely. An element of $\FF_2$ is said to be {\textit{primitive}} if it can be extended to a basis of $\FF_2$. A {\textit{basic triple}} of $\FF_2$ is an ordered triple
\begin{equation*}
	(X,Y,Z)\in\FF_2\times\FF_2\times\FF_2,
\end{equation*}
such that $(X,Y)$ is a basis of $\FF_2$ and $XYZ=id$ where $id$ is the identity element of $\FF_2$. The {\textit{automorphism group}} $\Aut(\FF_2)$ of $\FF_2$ acts transitively and freely on the set of bases, hence on the set of ordered triples, as well.

An element in $\FF_2$  induces an automorphism of $\FF_2$ obtained by taking conjugation of $\FF_2$ by this element. Such an automorphism is called an {\textit{inner-automorphism}}. The inner-automorphisms of $\FF_2$ form a normal subgroup of $\Aut(\FF_2)$, called the {\textit{inner-automorphism group}} and denoted by $\Inn(\FF_2)$. The quotient group 
\begin{equation*}
	\Aut(\FF_2)/\Inn(\FF_2),
\end{equation*}
is called the {\textit{outer-automorphism group}} of $\FF_2$ and denoted by $\Out(\FF_2)$.

Given a basis $(X,Y)$, we can associate to it an {\textit{elliptic involution}} $\mathfrak{e}_{(X,Y)}$ which is an automorphism of $\FF_2$ such that:
\begin{eqnarray*}
	&&\mathfrak{e}_{(X,Y)}(X)=X^{-1}\\
	&&\mathfrak{e}_{(X,Y)}(Y)=Y^{-1}
\end{eqnarray*}

Elliptic involutions associated to different bases differ by an inner automorphism. We denote by ${\Inn}^{\mathfrak{e}}(\FF_2)$ the subgroup of $\Aut(\FF_2)$ generated by elliptic involutions and $\Inn(\FF_2)$. It contains $\Inn(\FF_2)$ as an index two subgroup. The quotient group
\begin{equation*}
	{\Inn}^{\mathfrak{e}}(\FF_2)/\Inn(\FF_2),
\end{equation*}
is the center of $\Out(\FF_2)$. 

\begin{defn}
	An {\textit{ordered superbasis}} of $\FF_2$ is an ${\Inn}^{\mathfrak{e}}(\FF_2)$-orbit of basic triples.
\end{defn}
The $\Aut(\FF_2)$-action on basic triples induces an $\Out(\FF_2)$-action on ordered superbases. This action is transitive.

The symmetric group $\mathfrak{S}_3$ over $3$ symbols acts on the space of basic triples by the action generated by:
\begin{eqnarray*}
	(X,Y,Z)&\overset{(12)}\longmapsto&(Y^{-1},X^{-1},Z^{-1});\\
	(X,Y,Z)&\overset{(23)}\longmapsto&(X^{-1},Z^{-1},Y^{-1}).
\end{eqnarray*}
This $\mathfrak{S}_3$-action on basic triples commutes with the ${\Inn}^{\mathfrak{e}}(\FF_2)$-action, therefore it induces an $\mathfrak{S}_3$-action on ordered superbases.

\begin{defn}
	 An {\textit{unordered superbasis}} (or a {\textit{superbasis}} for short) of $\FF_2$ is an  $\mathfrak{S}_3$-orbit of ordered superbases. Two distinct unordered superbases are said to be {\textit{neighbors}} if they have representatives sharing two elements up to taking inverse.
\end{defn}

\begin{ex}
	The superbasis corresponding to $(X,Y,(XY)^{-1})$ and that corresponding to $(X,Y^{-1},YX^{-1})$ are neighbors.
\end{ex}
\begin{rmk}
	To simplify the notation, in the reminder of the paper, we will use $(X,Y,Z)$ to denote both an ordered triple and its corresponding superbasis. The meaning will be clear by the context.
\end{rmk}
A graph can be constructed from the  superbases of $\FF_2$. The vertices are the superbases and two vertices are connected by an edge if the corresponding superbases are neighbors. This graph is a trivalent tree. We call it the \textbf{\textit {tree of superbases}} and denote it by $\Sigma$. For our convenience, we also consider the metric on $\Sigma$ induced by setting each edge with length $1$. 

\begin{rmk}
	From the discussion, we can see that the group $\Out(\FF_2)$ is isomorphic to
	\begin{equation*}
		\mathbb{Z}_2\times((\mathbb{Z}_2\ast\mathbb{Z}_2\ast\mathbb{Z}_2)\rtimes\mathfrak{S}_3),
	\end{equation*}
	where $\mathbb{Z}_2=\mathbb{Z}/2\mathbb{Z}$ in the first factor corresponds to the elliptic involution.
\end{rmk}

\subsection{\texorpdfstring{$\SLtwoC$}{Lg}-character variety of \texorpdfstring{$\FF_2$}{Lg}}
Recall that the $\SLtwoC$-character variety $\chi$ of $\FF_2$ is the space of conjugacy classes of homomorphisms from $\FF_2$ to $\SLtwoC$. A classical result of Fricke and Vogt about $\chi$ is stated as follows:
\begin{thm}[Fricke \cite{fricke1896}, Vogt \cite{vogt1889}]
	Let $f:\SLtwoC\times\SLtwoC\rightarrow\CC$ be a regular function which is invariant under the diagonal action of $\SLtwoC$ by conjugation. There exists a polynomial function $F(x,y,z)\in\CC[x,y,z]$ such that: for any pair $(\xi,\eta)\in\SLtwoC\times\SLtwoC$, we have
	\begin{equation}
	f(\xi,\eta)=F(\mathrm{tr}(\xi),\mathrm{tr}(\eta),\mathrm{tr}(\xi\eta)).
	\end{equation}
	Furthermore, for all $(x,y,z)\in\CC^3$, there exists $(\xi,\eta)\in\SLtwoC\times\SLtwoC$ such that
	\begin{displaymath}
	\left[\begin{array}{l}
	x\\
	y\\
	z
	\end{array}\right]
	=\left[\begin{array}{l}
	\mathrm{tr}(\xi)\\
	\mathrm{tr}(\eta)\\
	\mathrm{tr}(\xi\eta)
	\end{array}\right].
	\end{displaymath}
	Conversely, if $x^2+y^2+z^2-xyz\neq 4$ and $(\xi,\eta)$, $(\xi',\eta')\in\SLtwoC\times\SLtwoC$ satisfy,
	\begin{displaymath}
	\left[\begin{array}{l}
	x\\
	y\\
	z
	\end{array}\right]
	=\left[\begin{array}{l}
	\mathrm{tr}(\xi)\\
	\mathrm{tr}(\eta)\\
	\mathrm{tr}(\xi\eta)
	\end{array}\right]
	=\left[\begin{array}{l}
	\mathrm{tr}(\xi')\\
	\mathrm{tr}(\eta')\\
	\mathrm{tr}(\xi'\eta')
	\end{array}\right],
	\end{displaymath}
	Then, there exists an element $g\in\SLtwoC$ such that $(\xi,\eta)=(g\xi'g^{-1},g\eta'g^{-1})$.
\end{thm}
Fixing a superbasis $(X,Y,Z)$, we can define a map 
\begin{equation*}
	\Phi:\Hom(\FF_2,\SLtwoC)\rightarrow\CC^3,
\end{equation*}
sending $\rho$ to $(x,y,z)=(\tr\rho(X),\tr\rho(Y),\tr\rho(Z))$.
\begin{defn}
	A representation $\rho:\FF\rightarrow\SLtwoC$ is said to be {\textit{reducible}} if its image preserves a non-trivial proper subspace when acting on $\CC^2$. A representation is said to be {\textit{irreducible}} if it is not reducible.
\end{defn}
It is well-known that a representation $\rho$ is reducible if and only if it is $\SLtwoC$-conjugate to a representation by upper triangular matrices in $\SLtwoC$. By using Fricke's trace identity, we have:
 	\begin{equation}\label{trace}
 	\mathrm{tr}([\xi,\eta])=\mathrm{tr}^2(\xi)+\mathrm{tr}^2(\eta)+\mathrm{tr}^2(\xi\eta)-\mathrm{tr}(\xi)\mathrm{tr}(\eta)\mathrm{tr}(\xi\eta)-2.
 	\end{equation}
Moreover, Nielson \cite{nielsen1917mathann} showed that all commutators of pairs of free generators of $\FF_2$ are conjugate to each other up to taking inverse. Therefore, we conclude that a representation is reducible if and only if its trace of commutator of a pair of generators equals $2$. By the above theorem, two irreducible representations $\rho_1$ and $\rho_2$ are conjugate to each other if and only if $\Phi(\rho_1)=\Phi(\rho_2)$.

\begin{rmk}
	Using this coordinate system, we can see that the elliptic involution acts trivially on the $\chi$. The $\mathfrak{S}_3$ acts as permutations of the three coordinates. Hence to see the $\Out(\FF_2)$-action on $\chi$, it is enough to study the $\ZZ_2\ast\ZZ_2\ast\ZZ_2$-action. This explains why we are interested in the tree of superbasis.
\end{rmk}

\subsection{Bowditch's Q-conditions}
In this section, we recall the work of Bowditch \cite{bowditch1998plms} and its generalization by Tan-Wong-Zhang \cite{tan-wong-zhang2008advm}.

Let us consider $\Sigma$ the trivalent tree of superbases associated to $\FF_2$. Let $V(\Sigma)$ and $E(\Sigma)$ denote the vertex set  and the edge set of $\Sigma$ respectively. It admits a proper embedding in the Poincar\'e disk $\DD$ as the dual graph of Farey tessellation $\mathcal{F}$ of $\DD$.

More precisely, the conformal map from $\DD$ to the upper half plane model of hyperbolic plane induces an identification of the boundary $S^1$ of $\DD$ with $\RR\cup\infty$. A point on $S^1$ is called {\textit{rational}} if it is mapped to $p/q$ under this identification, where $p$ and $q$ are coprime integers. As a convention $0$ is considered as $0/1$ and $\infty$ is considered as $1/0$. Two rational points $p/q$ and $s/t$ are said to be {\textit{Farey neighbors}} if $|pt-qs|=1$. The {\textit{Farey tessellation}} $\mathcal{F}$ of $\DD$ is given by connecting every pair of Farey neighbors by geodesics. It is an ideal triangulation of $\DD$. Therefore the dual graph is an infinite trivalent tree. 

Given any element $W$ of $\FF_2$, we denote by $[W]$ its $\Inn^\mathfrak{e}(\FF_2)$-orbit. We denote by $\Omega$ the set of all $[W]$ with $W$ primitive. There is an 1-1 correspondence between rational points on $S^1$ and $\Omega$ described in the following way. Let $(X_0,Y_0)$ be a basis of $\FF_2$. Then any primitive element can be written as a word of $X_0$ and $Y_0$ with integer powers. Consider the abelianzation of $\FF_2$ to $\ZZ^2$ sending $X_0$ to $(1,0)$ and $Y_0$ to $(0,1)$. Any primitive element becomes $(q,p)$ with $p$ and $q$ coprime. A theorem of Nielsen \cite{nielsen1917mathann} tells us that two primitive elements of $\FF_2$ are conjugate if and only if they have the same image in $\ZZ^2$. Therefore the rational number $p/q$ determines an $\Inn^\mathfrak{e}(\FF_2)$-orbit of a primitive element in $\FF_2$. This induces an identification between rational points on $S^1$ and $\Omega$. 

We say that two orbits $[X]$ and $[Y]$ are {\textit{Farey neighbors}} if their corresponding rational points are Farey neighbors. It is easy to check the following two facts:
\begin{enumerate}
	\item $[X]$ and $[Y]$ are Farey neighbors if and only if we can find $X'\in[X]$ and $Y'\in[Y]$ such that $(X',Y')$ is a basis of $\FF_2$;
	\item $[X]$, $[Y]$ and $[Z]$ are vertices of an ideal triangle in the complement of $\mathcal{F}$ if and only if we can find three elements, one in each orbit, which form an ordered triple.
\end{enumerate}
Hence, an ideal triangle in the complement of $\mathcal{F}$ corresponds to a unique superbasis of $\FF_2$. As a result, we find a map identifying the dual tree of $\mathcal{F}$ with the tree of superbases $\Sigma$. To simplify the notation, the dual tree of $\mathcal{F}$ will also be denoted by $\Sigma$.

\begin{rmk}
	The embedding of $\Sigma$ in $\mathbb{D}$ depends on the choice of $(X_0,Y_0)$. We will fix one embedding in the reminder of this paper.
\end{rmk}

Consider the complement of $\Sigma$ in $\DD$. Each connected component is bounded by a bi-infinite geodesic on $\Sigma$ asymptotic to a same rational point on $S^1$ along both directions. This gives a 1-1 correspondence between connected components of $\DD\setminus\Sigma$ and rational points on $S^1$, which moreover induces a labeling of connected components by $\Inn^\mathfrak{e}(\FF_2)$-orbits of primitive elements of $\FF_2$. Then each edge $e\in E(\Sigma)$ can be labeled by $(X,Y;Z,Z')$ where $[X]$ and $[Y]$ are the two connected components adjacent to $e$ and $[Z]$ and $[Z']$ are the two connected components only meeting $e$ at vertices. 
\begin{center}
	\includegraphics[scale=0.3]{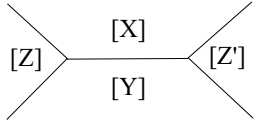}
	\captionof{figure}{Labelling an edge}
\end{center}

A representation $\rho$ induces a function defined on $\Omega$ sending $[X]$ to $x=\mathrm{tr}\rho(X)$ which we call the {\textit{trace function}} $f_\rho$ for $\rho$. The $f_\rho$-values of three regions meeting at any vertex satisfy the following relation:
\begin{equation}\label{commutator}
	x^2+y^2+z^2=xyz+\mu,
\end{equation}
where $\mu=2+\mathrm{tr}\rho([X,Y])$ for a basis $(X,Y)$.

Given four elements corresponding to an edge, the traces of their images under $\rho$ satisfy the following relation:
\begin{eqnarray}
	z'+z&=&xy,\\\label{growth}
	z'z&=&x^2+y^2-\mu
\end{eqnarray}
This allows us to  compute the traces of all primitive elements inductively from the trace of one superbasis.

A representation is said to be {\textit{type preserving}} if $\mu=0$. Let $\rho$ be a type preserving representation. In \cite{bowditch1998plms}, Bowditch studied the growth rate of $f_\rho$ when computing the traces of primitive elements using the above inductive method. In particular, he was interested in comparing $f_\rho$ with the Fibonacci function defined as follows.

Let $e=(X,Y;Z,Z')$ be a fixed edge of the tree $\Sigma$. We set $F_e([X])=F_e([Y])=1$ and $F_e([Z])=F_e([Z'])=2$. Then for a vertex $v_n=(X_n,Y_n,Z_n)$ at distance $n$ to $e$, we set $F_e([Z_n])=F_e([X_n])+F_e([Y_n])$ if the distance from $Z_n$ to $e$ is greater than the other two distances. In this way, we define a function $F_e:\Omega \rightarrow \NN$ which is called the {\textit{Fibonacci function}} with respect to the edge $e$.
\begin{center}
	\includegraphics[scale=0.4]{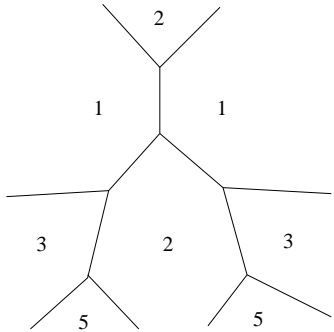}
	\captionof{figure}{Fibonacci function}
\end{center}

Let $f:\Omega\rightarrow[0,+\infty)$ be any function.
\begin{defn}
 The function $f$ is said to have an {\textit{upper Fibonacci bound}} if there exists a constant $K>0$ such that $f([X])\le KF_e([X])$ for all $[X]\in\Omega$;

 The function $f$ is said to have a {\textit{lower Fibonacci bound}} if there exists a constant $k>0$ such that $f([X])\ge kF_e([X])$ for all but finitely many $[X]\in\Omega$;

 The function $f$ is said to have a {\textit{Fibonacci growth}} if it has both lower and upper Fibonacci bounds.
\end{defn}
\begin{rmk}
	The Fibonacci function $F_e$ with respect to $e$ has Fibonacci growth with respect to the Fibonacci function $F_{e'}$ with respect to any other edge $e'$. Hence, the above definition is independent of the choice of $e$.
\end{rmk}
Let $f_\rho$ be the trace function for a type-preserving representations $\rho$. In \cite{bowditch1996blms} Bowditch proved:
\begin{thm}\label{BQ}
	The function $f_\rho$ satisfies:
	\begin{enumerate}
		\item $f_\rho^{-1}([-2,2])=\emptyset$;
		\item $|f_\rho|^{-1}([0,2])$ is finite.
	\end{enumerate} 
	if and only if the function $\log^+|f_\rho|$ has Fibonacci growth, where $|\cdot|$ stands for taking the modulus of a complex number and $\log^+:=\max\{0,\log\}$.
\end{thm}

In \cite{tan-wong-zhang2008advm}, Tan, Wong and Zhang generalized Bowditch's work by dropping the type preserving restriction. In particular, they proved that the above theorem is true for any representation $\rho$ with $\mu\neq4$.

\begin{defn}
	An irreducible representation $\rho$ satisfies {\textit{BQ-conditions}} if its trace function $f_\rho$ satisfies the following two conditions:
	\begin{enumerate}
		\item $f_\rho^{-1}([-2,2])=\emptyset$;
		\item $|f_\rho|^{-1}([0,2])$ is finite.
	\end{enumerate}
\end{defn}

\section{Edge orientations on \texorpdfstring{$\Sigma$}{Lg} and corresponding attracting subtrees}
\subsection{Basic Definitions}
Let us consider the tree of superbasis $\Sigma$.
\begin{defn}
	We say that $\Sigma$ is equipped with an {\textit{edge orientation}} if each of its edges is equipped with an orientation.
\end{defn}
One way to give an edge orientation on $\Sigma$ can be described as follows. Denote by $\widetilde{\Omega}$ the set:
\begin{equation*}
	\widetilde{\Omega}:=\{([W],v)\in\Omega\times V(\Sigma)\mid \textrm{$v$ is a vertex on the boundary of $[W]$}\}.
\end{equation*}  
Let $f$ be a function on $\widetilde{\Omega}$ with values in $\RR_{\ge0}$. It induces an edge orientation on $\Sigma$ in the following way: for the edge $e=(X,Y;Z,Z')$ with vertices $v=(X,Y,Z)$ and $v'=(X,Y,Z')$, 
\begin{itemize}
	\item if $|f([Z],v)|>|f([Z'],v')|$,  $e$ is oriented from $v$ to $v'$;
	\item if $|f([Z],v)|<|f([Z'],v')|$,  $e$ is oriented from $v'$ to $v$;
	\item if $|f([Z],v)|=|f([Z'],v')|$, we choose one of the two orientations arbitrarily.
\end{itemize}
\begin{rmk}
	A function on $\Omega$ induces a function on $\widetilde{\Omega}$ by pre-composing the projection from $\widetilde{\Omega}$ to $\Omega$. In this case, we call the latter  the {\textit{lift}} of the former.
\end{rmk}
\begin{rmk}
	The edge orientation induced by $f$ is not  unique if the equality holds at some edges. Below we will see that in our case the equality holds only at finitely many edges and the induced orientation is unique up to finitely many edges.
\end{rmk}
\begin{defn}
		A subgraph $\Sigma_0$ of $\Sigma$ is called an {\textit{attracting subtree}} of $\Sigma$ for the given edge orientation if it satisfies the following two conditions:
	\begin{itemize}
		\item it is connected;
		\item Every edge $e$ not contained in $\Sigma_0$ is oriented towards $\Sigma_0$.
	\end{itemize}
\end{defn}

\begin{rmk}
	Since $\Sigma$ is simply connected, for any edge orientation, the minimal attracting subtree, if it exists, is unique.
\end{rmk}

\subsection{Three edge orientations induced by an irreducible representation}
In this subsection, we show that an irreducible representation $\rho$ induces three functions on $\widetilde{\Omega}$ with values in $\RR_{\ge0}$ which in turn induces three edge orientations on $\Sigma$. The first two are the lifts of functions on $\Omega$, while the third one is not. In the following, we discuss them case by case.

\medskip

\noindent(I) \textit{The modulus of traces function.} To define the first function, we consider the trace function $f_\rho$ induced by $\rho$ and lift it to a function defined on $\widetilde{\Omega}$ which is denoted by $\widetilde{f}_\rho$. By post-composing by the modulus function on $\CC$, we obtain the {\textit{function of modulus of traces}} $|\widetilde{f}_\rho|$ on $\widetilde{\Omega}$. The corresponding edge orientation was first studied by Bowditch \cite{bowditch1998plms} for the case where $\rho$ is type preserving. By generalizing a result of Bowditch, Tan-Wong-Zhang proved the following result:

\begin{prop}[Tan-Wong-Zhang \cite{tan-wong-zhang2008advm}]\label{attractingsubtree1}

	If an irreducible representation $\rho$ satisfies the BQ-conditions, then there exists a compact attracting subtree in $\Sigma$ with respect to the edge orientation induced by $|\widetilde{f}_\rho|$. 
\end{prop}

\medskip

\noindent(II) \textit{The real translation distance function.} Consider the real translation distance function $a_\rho$ defined as follows:
\begin{eqnarray*}
        a_\rho:\Omega&\rightarrow&\RR_{\ge0}\\
       \left[W\right]&\mapsto& a(\rho(W))	
\end{eqnarray*}
where $a(\rho(W))$ is the real translation distance of $\rho(W)$ in $\HHH$ defined in the previous section. Its lift $\widetilde{a}_\rho$ is the {\textit{real translation distance function}} on $\widetilde{\Omega}$.

By our discussion in the previous section, the trace $\tr(\xi)$ of a loxodromic element $\xi\in\SLtwoC$ is
\begin{equation*}
	-2\cosh\frac{a(\xi)+i\alpha(\xi)}{2},
\end{equation*}  
where $a(\xi)+i\alpha(\xi)$ is the complex translation distance of $\xi$. Notice that for a pair of positive number $a$ and $\alpha$, the difference
\begin{equation*}
	\left|2\cosh\frac{a+i\alpha}{2}\right|-\exp\frac{a}{2}
\end{equation*}
is $O(\exp(-\frac{a}{2}))$ as $a$ goes to infinity. Therefore Proposition \ref{attractingsubtree1} above implies a similar result for $\widetilde{a}_\rho$:

\begin{prop}\label{attractingsubtree2}
	If an irreducible representation $\rho$ satisfies the BQ-conditions, then there exists a compact attracting subtree in $\Sigma$ with respect to the edge orientation induced by $\widetilde{a}_\rho$.
\end{prop}
\begin{proof}
	By Theorem 3.3 of \cite{tan-wong-zhang2008advm} and the fact that Fibonacci growth implies that the BQ-conditions are satisfied, a representation $\rho$ satisfies the BQ-conditions if and only if $\log^+|f_\rho|$ has Fibonacci growth. Let $v$ be a vertex of $\Sigma$. The Fibonacci growth of $\log^+|f_\rho|$ implies that the difference 
	\begin{equation*}
		||f_\rho|(W)-\exp\frac{a_\rho(W)}{2}|
	\end{equation*}
	 for $W$ with distance $N$ to $v$ is $O(\exp(-kN))$ as $N$ goes to infinity, where $k$ is the constant appearing in the lower Fibonacci growth inequality for $\log^+|f_\rho|$. Another consequence of the Fibonacci growth result is that for any compact subset in $\RR_{\ge0}$, its $\log^+|f_\rho|$-pre-image is finite. Combining these two facts, we can conclude that the edge orientation induced by $|\widetilde{f}_\rho|$ and that induced by $\widetilde{a}_\rho$ coincide on all but finitely many edges. Then this proposition follows from Proposition \ref{attractingsubtree1}.
\end{proof}

\medskip

\noindent(III) \textit{Angle function} The third function is called the {\textit{angle function}} denoted by $A_\rho$. For the study in this part, we assume moreover that the $\rho$-image of all primitive elements are loxodromic which is the case for $BQ$-representations. 

Let $([X],v)\in \widetilde{\Omega}$. Consider a representative $(X,Y,Z)$ of the superbasis corresponding to $v$. Since all primitive elements are sent to loxodromic elements by $\rho$, the axes of $\rho(X)$, $\rho(Y)$ and $\rho(Z)$ exist, and we can orient them, so that their translation directions are the positive directions respectively. We denote the three oriented axes by $\delta_X$, $\delta_Y$ and $\delta_Z$ respectively.  We denote by $\gamma_X$, $\gamma_Y$ and $\gamma_Z$ the axes of $\pi$-rotations $r(X)$, $r(Y)$ and $r(Z)$ respectively, where $r(X)$, $r(Y)$ and $r(Z)$ satisfy:
\begin{eqnarray*}
	\rho(X)&=&r(Y)r(Z),\\
	\rho(Y)&=&r(Z)r(X),\\
	\rho(Z)&=&-r(X)r(Y).
\end{eqnarray*}

Then  $A_\rho$ at $([X],v)$ is defined to be the complementary angle of the angle between the axes of $\rho(Y)$ and $\rho(Z)$. Notice that the value of $A_\rho$ depends on $[X]$, as well as $v$ which is different from the previous two cases. We will show:
\begin{prop}\label{attractingsubtree3}
	An irreducible representation $\rho$ satisfies the BQ-conditions, then there exists a compact attracting subtree in $\Sigma$ with respect to the edge orientation induced by $A_\rho$.
\end{prop}
Combining the above three propositions, we have the following corollary:
\begin{cor}\label{3orientation}
	If an irreducible representation $\rho$ satisfies BQ-condition, then there exists a compact subtree of $\Sigma$ which is attracting with respect to all three edge orientations.
\end{cor}

\subsection{Proof of Proposition \ref{attractingsubtree3}}
\subsubsection{Cosine rule for right-angled hexagons}
In Section VI of \cite{fenchel}, Fenchel gave an interpretation of the cosine rule and sine rule for a right-angled hexagon using what he called line matrices. They are called $\pi$-rotations in our work. We recall the cosine rule following his idea. As a convention, if two sides of $H$ are non-degenerate, we assume that they are non-collinear. We also assume that two adjacent sides of $H$ are not both degenerate.

Let us first assume that $H$ is non-degenerate. Let $s_1$, $s_2$, $s_3$, $s_4$, $s_5$ and $s_6$ denote its sides in a cyclic order. All indices in this part will be considered up to mod $6$. Let $\gamma_n$ denote the geodesic containing $s_n$. We choose the orientation for each $\gamma_n$ which is compatible with this cyclic order of $s_n$'s. Let $r_n$ denote the $\pi$-rotation with respect to $\gamma_n$. Since $\gamma_{n-1}$ and $\gamma_{n+1}$ are disjoint, the product $r_{n+1}r_{n-1}$ is a loxodromic element in $\SLtwoC$ whose orientated axis is $\gamma_n$. Let $l_n$ denote half of its complex translation distance. Then we have:
\begin{eqnarray*}
	\cosh l_n&=&-\frac{1}{2}\tr(r_{n+1}r_{n-1}),\\
	\sinh l_n&=&\frac{1}{2}i\,\tr(r_{n+1}r_nr_{n-1}).\\
\end{eqnarray*}
The cosine rule for $H$ is written as follows:
\begin{equation}\label{cosine}
	\cosh l_n=\cosh l_{n-2}\cosh l_{n+2}+\sinh l_{n-2}\sinh l_{n+2}\cosh l_{n+3}.
\end{equation}
Now let us consider the degenerate case. There are three cases:
\begin{itemize}
	\item there is one degenerate side;
	\item there are two degenerate sides which are non-adjacent;
	\item there are three alternate sides of $H$ which are degenerate, so that $H$ becomes a triangle.
\end{itemize}
If $s_n$ is degenerate, then $\gamma_n$ is the common perpendicular geodesic of $\gamma_{n-1}$ and $\gamma_{n+1}$. The orientation on $\gamma_n$ chosen  so that the angle counted from the positive direction of $\gamma_{n-1}$ to the positive direction of $\gamma_{n+1}$ is positive with value in $(0,\pi)$. Then the identity (\ref{cosine}) holds in the degenerate cases by replacing "$\cosh$" and "$\sinh$" by "$\cos$" and "$\sin$" respectively for the degenerate sides.

Next we would like to show that the identity (\ref{cosine}) still holds when we choose arbitrary combination of orientations of $\gamma_i$'s. To see this, we may rewrite (\ref{cosine}) using $\pi$-rotations as follows:
\begin{eqnarray*}
	&&-\frac{1}{2}\tr(r_{n+1}r_{n-1})-\frac{1}{4}\tr(r_{n-1}r_{n-3})\tr(r_{n-3}r_{n+1})\\
	&=&\frac{1}{8}\tr(r_{n-1}r_{n-2}r_{n-3})\tr(r_{n-3}r_{n+2}r_{n+1})\tr(r_{n-2}r_{n+2})
\end{eqnarray*}
When we change the orientation of $\gamma_n$, the corresponding $\pi$-rotation is changed from $r_n$ to $-r_n$. Since the $\pi$-rotation along each side appears either once in each term or twice in some terms, we can check that (\ref{cosine}) still holds with arbitrary sides orientations.

\subsubsection{right-angled hexagon for an ordered triple}
Let $\rho$ be a BQ-representation. Let $(X,Y,Z)$ be an ordered triple. Let $(r(X),r(Y),r(Z))$ be a triple of $\pi$-rotations satisfying:
\begin{eqnarray*}
	\rho(X)&=&r(Y)r(Z),\\
	\rho(Y)&=&r(Z)r(X),\\
	\rho(Z)&=&-r(X)r(Y).
\end{eqnarray*}
Let $\delta_X$, $\delta_Y$, $\delta_Z$, $\gamma_X$, $\gamma_Y$ and $\gamma_Z$ denote the oriented axes of $\rho(X)$, $\rho(Y)$, $\rho(Z)$, $r(X)$, $r(Y)$ and $r(Z)$ respectively. The orientations on $\delta_X$, $\delta_Y$ and $\delta_Z$ are chosen to be the same as the translation directions of actions of $\rho(X)$, $\rho(Y)$ and $\rho(Z)$ respectively. The six geodesics bound a right-angled hexagon denoted by $H(X,Y,Z)$. 
\begin{center}
	\includegraphics[scale=0.3]{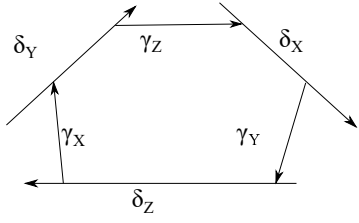}
	\captionof{figure}{Right-angled hexagon $H(X,Y,Z)$}
\end{center}

We denote by $r_\delta(X)$, $r_\delta(Y)$ and $r_\delta(Z)$ the three $\pi$-rotations along the oriented axes $\delta_X$, $\delta_Y$ and $\delta_Z$. We consider the complex transport distances:
\begin{eqnarray*}
	2l(\delta_X)&=&a(\delta_X)+i\alpha(\delta_X),\\
	2l(\delta_Y)&=&a(\delta_Y)+i\alpha(\delta_Y),\\
	2l(\delta_Z)&=&a(\delta_Z)+i\alpha(\delta_Z),\\
	2l(\gamma_X)&=&a(\gamma_X)+i\alpha(\gamma_X),\\
	2l(\gamma_Y)&=&a(\gamma_Y)+i\alpha(\gamma_Y),\\
	2l(\gamma_Z)&=&a(\gamma_Z)+i\alpha(\gamma_Z),
\end{eqnarray*}
of $r(Z)r(X)$, $r(Z)r(X)$, $r(X)r(Y)$, $r_\delta(Y)r_\delta(Z)$, $r_\delta(Z)r_\delta(X)$ and $r_\delta(X)r_\delta(Y)$, respectively. By Formula (\ref{tracerotation}) above, we have:
\begin{eqnarray*}
	2\cosh l(\delta_X)&=&-\tr(r(Y)r(Z)),\\
	2\cosh l(\delta_Y)&=&-\tr(r(Z)r(X)),\\
	2\cosh l(\delta_Z)&=&-\tr(r(X)r(Y)),\\
	2\cosh l(\gamma_X)&=&-\tr(r_\delta(Y)r_\delta(Z)),\\
	2\cosh l(\gamma_Y)&=&-\tr(r_\delta(Z)r_\delta(X)),\\
	2\cosh l(\gamma_Y)&=&-\tr(r_\delta(X)r_\delta(Y)).
\end{eqnarray*}
Using the cosine rule for $H(X,Y,Z)$, we have the following relations:
\begin{eqnarray}
	\cosh l(\gamma_X)=\frac{\cosh l(\delta_X)-\cosh l(\delta_Y)\cosh l(\delta_Z)}{\sinh l(\delta_Y)\sinh l(\delta_Z)},\label{cosine1}\\
	\cosh l(\gamma_Y)=\frac{\cosh l(\delta_Y)-\cosh l(\delta_Z)\cosh l(\delta_X)}{\sinh l(\delta_Z)\sinh l(\delta_X)},\label{cosine2}\\
	\cosh l(\gamma_Z)=\frac{\cosh l(\delta_Z)-\cosh l(\delta_X)\cosh l(\delta_Y)}{\sinh l(\delta_X)\sinh l(\delta_Y)}.\label{cosine3}	
\end{eqnarray}

Let $(x,y,z):=(\tr\rho(X),\tr\rho(Y),\tr\rho(Z))$. Then we have:
\begin{eqnarray*}
	x&=&-2\cosh l(\delta_X),\\
	y&=&-2\cosh l(\delta_Y),\\
	z&=&2\cosh l(\delta_Z).
\end{eqnarray*}
We recall that they satisfy Relation (\ref{commutator}):
\begin{equation*}
	x^2+y^2+z^2=xyz+\mu,
\end{equation*}
where $\mu$ only depends on $\rho$. We can rewrite it as:
\begin{equation*}
	\frac{x}{yz}+\frac{y}{xz}+\frac{z}{xy}=1+\frac{\mu}{xyz}	
\end{equation*} since $\rho$ satisfies the BQ-conditions and so $x,y,z \neq 0$.
Without loss of generality, we assume that $|z|\ge|y|\ge|x|$. Recall that the function $\log^+|f_\rho|$ has  Fibonacci growth. We consider the Fibonacci growth inequality for $\log^+|f_\rho|$ with respect to a initial edge $e$ on $\Sigma$. Let $k$ and $K$ be the two constant in the inequality for the lower bound and upper bound respectively. If $(X,Y,Z)$ is a vertex on $\Sigma$ with distance $N$ to $e$ and $N$ is large enough, then we have:
\begin{eqnarray*}
	|z|\ge|y|\ge\exp(k(N-1)).
\end{eqnarray*}
This implies:
\begin{lem}\label{cosh}
	The sum 
	\begin{equation*}
		\frac{1}{x}\left(\frac{z}{y}+\frac{y}{z}\right)
	\end{equation*} 
	converges to 1 uniformly with error $O(\exp(-kN))$ as $N$ goes to infinity.
\end{lem}

This in turn implies:
\begin{lem}\label{sinh}
	The quantity 
	\begin{equation*}
		\frac{-1}{2\sinh l(\delta_X)}\left(\frac{z}{y}-\frac{y}{z}\right)
	\end{equation*}
	converges to 1 uniformly with error $O(\exp(-kN))$ as $N$ goes to infinity.
\end{lem}
\begin{proof}
	\begin{eqnarray*}
	   &&\frac{1}{4\sinh^2 l(\delta_X)}\left(\frac{z}{y}-\frac{y}{z}\right)^2-1\\
	   &=&\frac{1}{4\sinh^2 l(\delta_X)}\left(\left(\frac{z}{y}-\frac{y}{z}\right)^2-4\sinh^2 l(\delta_X)\right)\\
	   &=&\frac{x^2}{4\sinh^2 l(\delta_X)}\left(\frac{1}{x^2}\left(\frac{z}{y}-\frac{y}{z}\right)^2-\frac{4\sinh^2 l(\delta_X)}{x^2}\right)\\	   
	   &=&\coth^2 l(\delta_X)\left(\frac{1}{x^2}\left(\frac{y}{z}+\frac{z}{y}\right)^2-1\right).
	\end{eqnarray*}
	
	Notice that the zeros for the functions "$\cosh$" and "$\sinh$"  are $(\pi/2+n\pi)i$, $(\pi+2n\pi)i$ and $2n\pi i$. If $\rho$ is  a BQ-representation, then the values of $\coth l(\delta_X)$ are bounded away from $0$ and $\infty$ uniformly for all primitive elements $X$. By Lemma \ref{cosh}, the above quantity converges to $0$ uniformly with error $O(\exp(-kN))$ as $N$ goes to infinity.
	
	Notice that the hyperbolic cosine function is injective on the subset of $\CC$ consisting of complex numbers whose real part is positive and imaginary part is in $[0,2\pi)]$.  By Lemma \ref{cosh} and the facts that $|z|\ge|y|$ and $x=-2\cosh l(\delta_X)$, we obtain the lemma.
\end{proof}
To finish the proof of \ref{attractingsubtree3}, we prove the following proposition:
\begin{prop}
	The triple $(l(\gamma(X)),l(\gamma(Y)),l(\gamma(Z)))$ converges to $(i\pi, i\pi, 0)$ uniformly with error $O(\exp(-kN))$ as $N$ goes to $\infty$.
\end{prop}
\begin{proof}
	It is enough to show the uniform convergence of 
	\begin{equation*}
		(\cosh l(\gamma_X), \cosh l(\gamma_Y), \cosh l(\gamma_Z))
	\end{equation*}
	to $(-1,-1,1)$ with the same error control as $N$ goes to $\infty$.
	
	\medskip

	\noindent (I) For $\cosh l(\gamma_X)$, we consider Formula (\ref{cosine1}):
	\begin{eqnarray*}
		&&\cosh l(\gamma_X)\\
		&=&\frac{\cosh l(\delta_X)-\cosh l(\delta_Y)\cosh l(\delta_Z)}{\sinh l(\delta_Y)\sinh l(\delta_Z)}\\
		&=&\frac{\cosh l(\delta_X)}{\sinh l(\delta_Y)\sinh l(\delta_Z)}-\frac{\cosh l(\delta_Y)\cosh l(\delta_Z)}{\sinh l(\delta_Y)\sinh l(\delta_Z)}.
	\end{eqnarray*}
	The proof of Proposition \ref{attractingsubtree2} shows that the real translation distance grows uniformly to infinity as $N$ goes to infinity. When $N$ is big, by our assumption, we have $\sinh l(\delta_Z)$ and $\cosh l(\delta_Z)$ (resp. $\sinh l(\delta_Y)$ and $\cosh l(\delta_Y)$) are close to $\exp(l(\delta_Z))/2$ (resp. $\exp(l(\delta_Z))/2$) with difference $O(\exp(-kN))$ as $N$ goes to infinity. Hence, we conclude that $\cosh l(\gamma_X)$ converges to $-1$ uniformly with error $O(\exp(-kN))$ as $N$ goes to infinity
	
	\medskip
	
	\noindent (II) For $\cosh l(\gamma_Y)$, we us  formula (\ref{cosine2}):
	\begin{eqnarray*}
		&&\cosh l(\gamma_Y)\\
		&=&\frac{\cosh l(\delta_Y)-\cosh l(\delta_Z)\cosh l(\delta_X)}{\sinh l(\delta_Z)\sinh l(\delta_X)}\\
		&=&\left(\frac{\cosh l(\delta_Y)}{\cosh l(\delta_Z)\cosh l(\delta_X)}-1\right)\frac{\cosh l(\delta_Z)\cosh l(\delta_X)}{\sinh l(\delta_Z)\sinh l(\delta_X)}\\
		&=&\left(\frac{2y}{xz}-1\right)\frac{\cosh l(\delta_Z)\cosh l(\delta_X)}{\sinh l(\delta_Z)\sinh l(\delta_X)}\\
		&=&\left(\left(\frac{z}{xy}+\frac{y}{xz}\right)-1+\left(\frac{y}{xz}-\frac{z}{xy}\right)\right)\frac{\cosh l(\delta_Z)\cosh l(\delta_X)}{\sinh l(\delta_Z)\sinh l(\delta_X)}.
	\end{eqnarray*}
	By Lemma \ref{cosh} and Lemma \ref{sinh}, we can conclude that $\cosh l(\gamma_Y)$ converges to $-1$ uniformly with error $O(\exp(-kN))$ as $N$ goes to infinity. 
	
	\medskip
	
	(III) The proof for the convergence of $\cosh l(\gamma_Z)$ is similar to that for $\cosh l(\gamma_Y)$:
	\begin{eqnarray*}
		&&\cosh l(\gamma_Z)\\
		&=&\frac{\cosh l(\delta_Z)-\cosh l(\delta_X)\cosh l(\delta_Y)}{\sinh l(\delta_X)\sinh l(\delta_Y)}\\
		&=&\left(\frac{\cosh l(\delta_Z)}{\cosh l(\delta_X)\cosh l(\delta_Y)}-1\right)\frac{\cosh l(\delta_X)\cosh l(\delta_Y)}{\sinh l(\delta_X)\sinh l(\delta_Y)}\\
		&=&\left(\frac{2z}{xy}-1\right)\frac{\cosh l(\delta_X)\cosh l(\delta_Y)}{\sinh l(\delta_X)\sinh l(\delta_Y)}\\
		&=&\left(\left(\frac{z}{xy}+\frac{y}{xz}\right)-1+\left(\frac{z}{xy}-\frac{y}{xz}\right)\right)\frac{\cosh l(\delta_Z)\cosh l(\delta_X)}{\sinh l(\delta_Z)\sinh l(\delta_X)}.
	\end{eqnarray*}
	By Lemma \ref{cosh} and Lemma \ref{sinh}, we can conclude that $\cosh l(\gamma_Z)$ converges to $1$ uniformly with error $O(\exp(-kN))$ as $N$ goes to infinity. 
\end{proof}
\begin{rmk}
	Using the above proposition, in \cite{leexu}, Lee and Xu prove the equivalence between the BQ-conditions and the primitive stability for $\PSLtwoC$-representations of $\FF_2$. The latter was first defined and studied by Minsky in \cite{minsky} where he studied the action of $\Out(\FF_n)$ on the $\PSLtwoC$-character variety of $\FF_n$.
\end{rmk}

\section{The Fermat point of a hyperbolic triangle}
In Euclidean geometry, we may find different special points associated to a triangle satisfying different properties. We are interested in the one which is called the Fermat point. More precisely, consider a triangle in the Euclidean plane. Then its Fermat point is the one which realizes the minimum of the sum of distances to its three vertices, among all points in the plane. In this section, we introduce its hyperbolic counterpart.  

\subsection{Convexity lemma in hyperbolic geometry}
Before going further, let us first recall the convexity property of the distance function in hyperbolic geometry which plays an important role in most of the proofs in the reminder of the paper. For any two points $p$ and $q$ in $\mathbb{H}^n$, we denote by $pq$ the geodesic segment connecting them and by $|pq|$ the hyperbolic length of this segment.
\begin{lem1}[Convexity Lemma]
	Let $p_1$, $p_2$, $p_3$ and $p_4$ be four points in $n$-dimensional hyperbolic space $\mathbb{H}^n$. Let $p$ and $p'$ be two points on $p_1p_2$ and $p_3p_4$ respectively, such that $|p_1p|/|p_1p_2|=|p_4p'|/|p_3p_4|=t$ for some $t\in(0,1)$. Then,
	\begin{equation}
	|pp'|\le (1-t)|p_1p_4|+t|p_2p_3|.
	\end{equation}
	Moreover equality is realized if and only if all four points lie on the same geodesic and $q$ is between $p$ and $p'$, where $q$ is on $p_1p_3$ such that $|p_1q|/|p_1p_3|=t$.
\end{lem1}
This is a standard result in the differential geometry of spaces with non-positive curvature. To make the paper self-contained, we give a proof of this lemma in the appendix. As a convention, by the convexity lemma we will always refer to this lemma.

\subsection{Compact hyperbolic triangles}
A {\textit{compact triangle}} in $\HH$ is a triangle whose sides are geodesic segments with finite lengths. We will not consider the degenerate case, where the three vertices are  collinear. All indices in this section are considered up to mod $3$. The following definitions will be useful in the subsequent discussions:
\begin{defn}
	A triangle in $\HH$ is called {\textit{$2\pi/3$-acute}} if all three internal angles are strictly smaller than $2\pi/3$; otherwise, it will be called {\textit{$2\pi/3$-obtuse}}.
\end{defn}
\begin{rmk}
	In particular, by our definition, a triangle with an internal angle equal to $2\pi/3$ is $2\pi/3$-obtuse.
\end{rmk}

Let $\Delta$ be a compact triangle in the hyperbolic plane $\HH$ with vertices $v_1$, $v_2$ and $v_3$. Its complement in $\HH$ has two connected components, each of which is an open set. We consider the bounded one and call it the {\textit{$\Delta$-domain}}.

\begin{defn}
	A point $p$ in $\HH$ distinct from $v_1$, $v_2$ and $v_3$ is said to be a {\textit{balanced}} point of $\Delta$ if it has the following property:
	\begin{equation}\label{balance}
	\angle v_1pv_2=\angle v_2pv_3=\angle v_3pv_1=\frac{2\pi}{3}.
	\end{equation}  
\end{defn}
We denote by $\widetilde{p}$ a balanced point.
\begin{prop}\label{balancedp}
	The triangle $\Delta$ admits a balanced point if and only if it is $2\pi/3$-acute. Moreover if it exists, it is unique and contained in the $\Delta$-domain.
\end{prop}
\begin{proof}
	We notice that if a point $p$ is contained in the complement of the $\Delta$-domain, then we have $\angle v_{j-1}pv_{j}+\angle v_{j}pv_{j+1}=\angle v_{j-1}pv_{j+1}$ for some $j$. Therefore to prove this proposition, it is enough to consider the points contained in the $\Delta$-domain.
	
	To prove the "only if" part, we assume that $\Delta$ is $2\pi/3$-obtuse and $\angle v_2v_1v_3\ge 2\pi/3$. Let $p$ be a point in the $\Delta$-domain. Let $q$ denote the intersection point between the geodesic containing $v_2p$ and the side $v_1v_3$.
	 
	\begin{center}
		\includegraphics[scale=1.0]{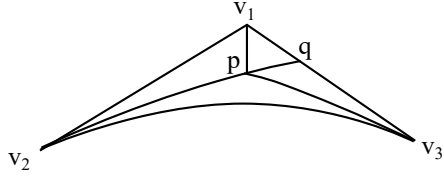}
		\captionof{figure}{$2\pi/3$-obtuse}
	\end{center}
	
	We then have the following relation:
	\begin{equation}
		\angle v_2pv_3>\angle v_2qv_3>\angle v_2v_1v_3\ge\frac{2\pi}{3}.
	\end{equation}
	Therefore $p$ cannot be balanced, hence the " only if " part.
	
	Now we assume that $\Delta$ is $2\pi/3$-acute. We may assume that $\angle v_2v_3v_1<\pi/3$. Consider the side $v_2v_3$. For any angle $\alpha\in(0,\pi)$. We say that a path in $\HH$ is $\alpha$-equiangular with respect to $v_2v_3$ if  $\angle v_2pv_3$ is constant (say $=\alpha$) when $p$ moves along this path. One may check that for each angle $\alpha$, there are two equiangular paths, symmetric with respect to $v_2v_3$. Moreover each one of them is an open path with endpoints at $v_2$ and $v_3$, and bounds a convex subset together with $v_2v_3$. Since we have:
	\begin{eqnarray*}
		\angle v_2v_3v_1&<&\pi/3,\\
		\angle v_1v_2v_3&<&2\pi/3,
	\end{eqnarray*}	one of the $2\pi/3$-equiangular paths of $v_2v_3$ intersects $v_1v_3$. For same reason,  one of the $2\pi/3$-equiangular paths with respect to $v_1v_3$ intersects $v_2v_3$. Therefore these two paths intersect in the $\Delta$-domain. The intersection point is a balanced point for $\Delta$. 
	
	\begin{center}
		\includegraphics[scale=1.0]{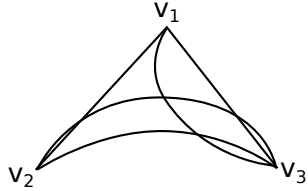}
		\captionof{figure}{$2\pi/3$-acute}
	\end{center}
	
	Let $\widetilde{p}$ be a balanced point of $\Delta$. Notice that the segments $\widetilde{p}v_1$, $\widetilde{p}v_2$ and $\widetilde{p}v_3$ separate $\Delta$ into three $2\pi/3$-obtuse triangles. Using the same argument as in the " only if " part, we can show that $\widetilde{p}$ is the only balanced point.
\end{proof}

\subsection{The Fermat Point of a compact triangle}
Consider $\Delta$ the compact triangle in the hyperbolic plane $\HH$ with vertices $v_1$, $v_2$ and $v_3$. We define a function $L:\HH\rightarrow\RR_{>0}$ by sending each point $p$ to $L(p)=|v_1p|+|v_2p|+|v_3p|$.
\begin{defn}
	The {\textit{Fermat point}} of $\Delta$ is a point in $\HH$ realizing the minimum of the function $L$.
\end{defn}

\begin{thm}\label{Fermat}
	The Fermat point of the triangle $\Delta$ exists and is unique. Moreover,
	\begin{enumerate}
		\item if $\Delta$ is $2\pi/3$-acute, then its balanced point is its Fermat point;
		\item if $\Delta$ is $2\pi/3$-obtuse, then its Fermat point is the vertex with the biggest internal angle.
	\end{enumerate}
\end{thm}

We first prove the following lemma:
\begin{lem}\label{out}
	If $p$ lies outside the closure of $\Delta$-domain, we can always find a point $q$ on $\Delta$, such that $L(p)>L(q)$.
\end{lem}
\begin{proof}
	Assume that $p$ lies outside $\Delta$. There are two possible cases:
	\begin{enumerate}
		\item there exists $j$ such that $pv_j$ intersects $v_{j-1}v_{j+1}$;
		\item there exists $j$ such that the vertex $v_j$ lies in the $\Delta pv_{j-1}v_{j+1}$-domain.
	\end{enumerate}
	
	Without loss of generality, in the following discussion, we assume that $j=3$ in both cases. In the first case, we have
	\begin{equation*}
		|pp_3|>|qp_3|.
	\end{equation*}
		
	\begin{center}
		\includegraphics[scale=1.0]{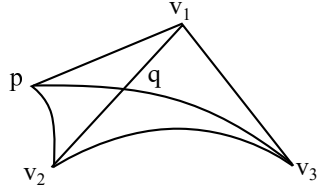}
		\captionof{figure}{First case}
	\end{center}
	
	At the same time, we have the triangular inequality for the triangle $\Delta pv_1v_2$:
	\begin{equation*}
	    |pv_1|+|pv_2|>|v_1v_2|=|qv_1|+|qv_2|.
	\end{equation*}
	Hence, we have $L(p)>L(q)$ in the first case.
		
\medskip
		
	In the second case, we have:
	\begin{equation*}
		\angle pv_3v_1+\angle pv_3v_2+\angle v_1v_3v_1=2\pi.
	\end{equation*}
	\begin{center}
		\includegraphics[scale=1.0]{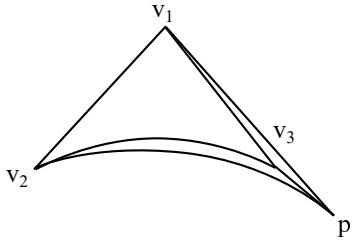}
		\captionof{figure}{Second case}
	\end{center}
	
	Therefore, either $\angle pv_3v_1$ or $\angle pv_3v_2$ is bigger than $\pi/2$. We assume that it is $\angle pv_3v_1$. Then we have:
	\begin{equation*}
	|v_1v_3|<|pv_1|.
	\end{equation*}
	On the other hand, by the triangle inequality for $\Delta pv_3v_2$, we have:
	\begin{equation*}
	|v_3v_2|<|pv_3|+|pv_2|.
	\end{equation*}
    Therefore, we have $L(p)>L(v_3)$ in the second case.
\end{proof}

\medskip

\begin{proof}[Proof of Theorem \ref{Fermat}]
	By the above lemma, to show the existence and uniqueness of the minimum point for the function $L$, it is enough to consider the points contained in the closure of the $\Delta$-domain. We prove the theorem by showing that in each case, the point mentioned in the statement of the theorem has strictly smaller $L$-value than any other point in the closure of the $\Delta$-domain.
	
	\smallskip
	
	\noindent \textbf{Case I: $\Delta$ is $2\pi/3$-obtuse.} Without loss of generality, let $v_1$ be the vertex of $\Delta$ with interior angle bigger than  or equal to $2\pi/3$. Let $p$ be a point in the closure of $\Delta$-domain different from $v_1$. Then the theorem is equivalent to saying that for any point $p$ in the $\Delta$-domain different from $v_1$ satisfies:
	\begin{equation*}
		L(p)>L(v_1).
	\end{equation*} 
	
	Let us consider the angles $\angle v_2v_1p$ and $\angle v_3v_1p$ and denote them by $\alpha_p$ and $\beta_p$ respectively for simplification of notation. By assumption, we have
	\begin{equation*}
		\alpha_p+\beta_p=\angle v_2v_1v_3\ge\frac{2\pi}{3}.
	\end{equation*}
	\begin{center}
		\includegraphics[scale=1.0]{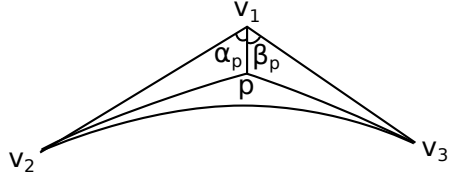}
		\captionof{figure}{$\alpha_p$ and $\beta_p$}
	\end{center}

	We first discuss the case where one of them is bigger or equal to $\pi/2$. Without loss of generality, we assume that $\alpha_p\ge\pi/2$. Then it is easy to check:
	\begin{eqnarray*}
		|pv_2|&\ge&|v_1v_2|;\\
		|pv_1|+|pv_3|&\ge&|v_1v_3|.
	\end{eqnarray*} 
	Moreover the two equalities can be realized at the same time if and only if $p=v_1$. Therefore, if $p\neq v_1$, we have
	\begin{equation*}
		L(p)>L(v_1).
	\end{equation*}	
	
	Secondly, we observe that if $|pv_1|\ge|v_1v_2|$ or $|pv_1|\ge|v_1v_3|$, since $|v_2v_3|$ is strictly bigger than $|v_1v_2|$ and $|v_1v_3|$, we have again
	\begin{equation*}
	L(p)>L(v_1).
	\end{equation*}	
	
	Now we discuss the remaining case. We may assume that
	\begin{eqnarray*}
	&&\alpha_p<\pi/2;\\
	&&\beta_p<\pi/2;\\
	&&|pv_1|<|v_1v_2|;\\
	&&|pv_1|<|v_1v_3|.
	\end{eqnarray*}
	By the cosine rule for $\Delta v_1pv_2$ and $\Delta v_1pv_3$, we have the following two identities:
	\begin{eqnarray}
	\cosh |pv_2|&=&\cosh|v_1v_2|\cosh|pv_1|-\sinh|v_1v_2|\sinh|pv_1|\cos\alpha_p;\label{left}\\
	\cosh |pv_3|&=&\cosh|v_1v_3|\cosh|pv_1|-\sinh|v_1v_3|\sinh|pv_1|\cos\beta_p.\label{right}
	\end{eqnarray}
	Our goal is to show that there is $t\in(0,1)$ such that
	\begin{eqnarray}
	\cosh|pv_2|&\ge&\cosh((1-t)|pv_1|)\cosh(|v_1v_2|-t|pv_1|);\label{left1}\\
	\cosh|pv_3|&\ge&\cosh (t|pv_1|)\cosh(|v_1v_3|-(1-t)|pv_1|).\label{right1}
	\end{eqnarray}
	To simplify the notation, we set
	\begin{eqnarray*}
		&&a'=|v_1v_2|;\\
		&&b'=|v_1v_3|;\\
		&&a=|pv_2|;\\
		&&b=|pv_3|;\\
		&&c=|pv_1|.
	\end{eqnarray*}
	We rewrite the right hand side of (\ref{left}) as follows:
	\begin{eqnarray*}
		&&\cosh a\\
		&=&\cosh a'\cosh(tc+(1-t)c)-\sinh a'\sinh(tc+(1-t)c)\cos\alpha_p\\
		&=&\cosh a'\Big(\cosh (tc)\cosh((1-t)c)+\sinh (tc)\sinh((1-t)c)\Big)-\\
		&&-\sinh a'\Big(\sinh (tc)\cosh((1-t)c)+\cosh (tc)\sinh((1-t)c)\Big)\cos\alpha_p\\
		&=&\cosh((1-t)c)\cosh(a'-tc)+(1-\cos\alpha_p)\sinh (tc)\sinh(a'+(1-t)c)-\\
		&&-\cos\alpha_p\sinh((1-t)c)\sinh(a'-tc).
	\end{eqnarray*}
	Similarly, the right hand side of (\ref{right}) can be rewritten as:
	\begin{eqnarray*}
		&&\cosh b\\
		&=&\cosh (tc)\cosh(b'-(1-t)c)+\cos\beta_p\sinh((1-t)c)\sinh(b'+tc)\\
		&&-(1-\cos\beta_p)\sinh(tc)\sinh(b'-(1-t)c).
	\end{eqnarray*}
		
	To show (\ref{left1}) and (\ref{right1}) hold, it is enough to show that there exists $t\in(0,1)$ such that:
	\begin{equation}\label{middle}
	0<\frac{\cos\alpha_p}{1-\cos\alpha_p}\le\frac{\sinh(tc)}{\sinh((1-t)c)}\le\frac{1-\cos\beta_p}{\cos\beta_p}<\infty.
	\end{equation} 
	If we remove the middle term, the remaining relations hold. To see this, we notice that
	\begin{eqnarray*}
		\frac{\pi}{6}<\alpha_p<\frac{\pi}{2},\\
		\frac{\pi}{6}<\beta_p<\frac{\pi}{2},
	\end{eqnarray*}
	hence, the above relation is equivalent to
	\begin{equation*}
		\cos\alpha_p+\cos\beta_p\le1,
	\end{equation*}
	which moreover can be written as
	\begin{equation*}
		2\cos\frac{\alpha_p+\beta_p}{2}\cos\frac{\alpha_p-\beta_p}{2}\le1.
	\end{equation*}
	The above inequality holds, since we have:
	\begin{equation*}
		\frac{2\pi}{3}\le\alpha_p+\beta_p\le\pi.
	\end{equation*}

	On the other hand, the function
	\begin{equation*}
		\frac{\sinh(tc)}{\sinh((1-t)c)}
	\end{equation*}
	has image $(0,\infty)$ for $t\in(0,1)$ and is monotonically increasing as $t$ grows. Hence we have the existence of $t$ for (\ref{middle}) to hold. Using the inequalities (\ref{left1}) and (\ref{right1}) and the monotonicity of the hyperbolic cosine function, we obtain moreover that for some $t\in(0,1)$:
	\begin{eqnarray*}
		a&\ge& a'-tc;\\
		b&\ge& b'-(1-t)c,
	\end{eqnarray*}
	where the two equalities hold if and only if when $p=v_1$. In that case, we have $c=0$. Hence we have
	\begin{equation*}
		L(p)=a+b+c>a'+b'=L(p_1),
	\end{equation*}
	for any $p$ in the closure of $\Delta$-domain different from $v_1$.

	\medskip
    
    \noindent\textbf{Case II: $\Delta$ is $2\pi/3$-acute.} We first prove the following lemma:
    \begin{lem}\label{anglehyp2}
    	Let $q$ be a point in $\HH$. Let $\mathsf{r}_1$, $\mathsf{r}_2$ and $\mathsf{r}_3$ be three distinct rays in $\HH$	starting at $q$ such that the angle between any two of them is $2\pi/3$. Let $q_1$, $q_2$ and $q_3$ be three points on $\mathsf{r}_1$, $\mathsf{r}_2$ and $\mathsf{r}_3$ respectively. Then,
    	\begin{equation*}
    	|qq_1|+|qq_2|+|qq_3|<|q_1q_2|+|q_1q_3|.
    	\end{equation*}
    \end{lem}
    \begin{proof}[Proof of Lemma \ref{anglehyp2}]
    	The proof uses  hyperbolic trigonometry as follows.
    	\begin{center}
    		\includegraphics[scale=0.6]{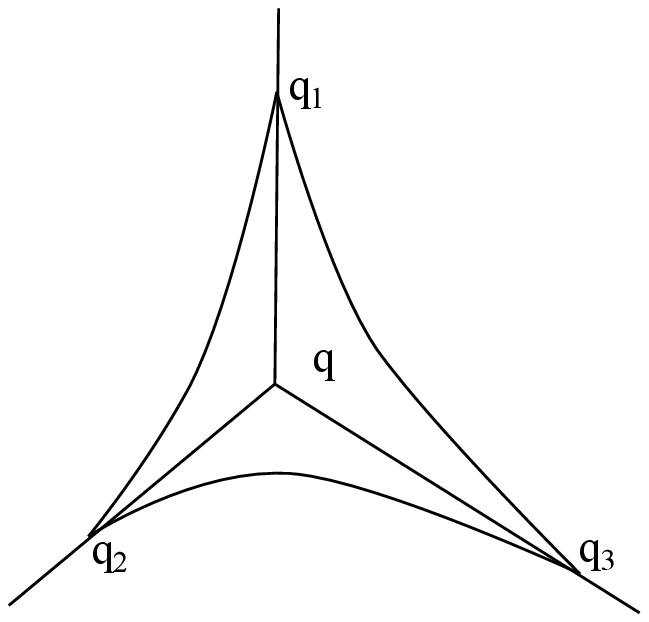}
    		\captionof{figure}{}
    	\end{center}
    	The cosine rule for $\Delta qq_1q_2$ and $\Delta qq_1q_3$ gives the following two identities:
    	\begin{eqnarray*}
    	\cosh|q_1q_2|=\cosh|qq_1|\cosh|qq_2|-\sinh|qq_1|\sinh|qq_2|\cos\frac{2\pi}{3};\\
    	\cosh|q_1q_3|=\cosh|qq_1|\cosh|qq_3|-\sinh|qq_1|\sinh|qq_3|\cos\frac{2\pi}{3}.
    	\end{eqnarray*}
    	Equivalently, we have:	
    	\begin{eqnarray*}
    	\cosh|q_1q_2|=\cosh|qq_1|\cosh|qq_2|+\frac{1}{2}\sinh|qq_1|\sinh|qq_2|;\\
    	\cosh|q_1q_3|=\cosh|qq_1|\cosh|qq_3|+\frac{1}{2}\sinh|qq_1|\sinh|qq_3|.
    	\end{eqnarray*}
    	Therefore we have:
    	\begin{eqnarray*}
    	\cosh|q_1q_2|=\cosh^2(\frac{1}{2}|qq_1|)\cosh|qq_2|+\sinh^2(\frac{1}{2}|qq_1|)\cosh|qq_2|\nonumber\\
    	+\sinh(\frac{1}{2}|qq_1|)\sinh|qq_2|\cosh(\frac{1}{2}|qq_1|);\\
    	\cosh|q_1q_3|=\cosh^2(\frac{1}{2}|qq_1|)\cosh|qq_3|+\sinh^2(\frac{1}{2}|qq_1|)\cosh|qq_3|\nonumber\\
    	+\sinh(\frac{1}{2}|qq_1|)\sinh|qq_3|\cosh(\frac{1}{2}|qq_1|).
    	\end{eqnarray*}
    	This implies:
    	\begin{eqnarray*}
    	\cosh|q_1q_2|=\cosh(\frac{1}{2}|qq_1|)\cosh(\frac{1}{2}|qq_1|+|qq_2|)+\sinh^2(\frac{1}{2}|qq_1|)\cosh|qq_2|;\\
    	\cosh|q_1q_3|=\cosh(\frac{1}{2}|qq_1|)\cosh(\frac{1}{2}|qq_1|+|qq_3|)+\sinh^2(\frac{1}{2}|qq_1|)\cosh|qq_3|.
    	\end{eqnarray*}
    	Since the hyperbolic cosine function is positve and strictly monotonically increasing on $\RR_{>0}$, we have the following two inequalities:
    	\begin{eqnarray*}
    	|q_1q_2|>\frac{1}{2}|qq_1|+|qq_2|;\\
    	|q_1q_3|>\frac{1}{2}|qq_1|+|qq_3|.
    	\end{eqnarray*}
    	
    	Taking the sum of the above two inequalities complete the proof of the lemma.
	\end{proof}
    
Now consider $\Delta$  and its balanced point $\widetilde{p}$. The three segments $v_1\widetilde{p}$, $v_2\widetilde{p}$ and $v_3\widetilde{p}$ separate the closure of the $\Delta$-domain into three closed subsets bounded by triangles $\Delta v_1v_2\widetilde{p}$, $\Delta v_2v_3\widetilde{p}$ and $\Delta v_1v_3\widetilde{p}$ respectively. It is sufficient to discuss the case where $p$ is contained in one of the three closed subsets. Without loss of generality, we assume $p$ is in the part bounded by $\Delta v_1v_2\widetilde{p}$. Moreover,  we may assume that $v_3{p}$ intersects $v_1\widetilde{p}$ at $q$, see figure below.

\begin{center}
	\includegraphics[scale=1.0]{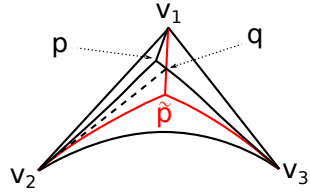}
	\captionof{figure}{$2\pi/3$-acute}
\end{center}

Since the triangle $\Delta v_1v_2\widetilde{p}$ is $2\pi/3$-obtuse, by the discussion for Case I above, we have
\begin{equation*}
	|v_1q|+|v_2q|\le|v_1p|+|v_2p|+|pq|,
\end{equation*}
where equality holds if and only if $p=q$. On the other hand, by  Lemma \ref{anglehyp2}, we have
\begin{equation*}
	|\widetilde{p}v_2|+|\widetilde{p}v_3|+|\widetilde{p}q|+|qv_1|\le|qv_2|+|qv_3|+|qv_1|,
\end{equation*}
where equality holds if and only if $q=\widetilde{p}$. This implies that $L(\widetilde{p})\le L(p)$ and  equality holds if and only if $p=\widetilde{p}$. Hence, $\widetilde{p}$ is the Fermat point of $\Delta$.
\end{proof}    
\begin{rmk}\label{rmk}
	Since the closed set bounded by $\Delta$ is compact and the function $L$ is continuous, another way to show the existence of the Fermat point is by using the extreme value theorem. Then  we have uniqueness by the convexity lemma. Furthermore, by Lemma \ref{anglehyp2}, for each case, only the point mentioned in the theorem can be a local minimum for $L$ from which we obtain the theorem.
\end{rmk}

\section{The Fermat Point and Steiner tree of a  triple of geodesics in \texorpdfstring{$\HHH$}{Lg}}
We would like to generalize the definition and discussion above to the Fermat point of a triple of geodesics in $\HHH$. More precisely, given three geodesics in $\HHH$, the Fermat point is defined to be the point realizing the minimum of the sum of the distances to the three geodesics among all points in $\HHH$. With the help of the Fermat point, we can find a graph connecting the three given geodesics with shortest length as possible, which will be called the Fermat tree.

 We will also be interested in the question whether the length of the double by the involution of $\HHH$ determined by one of the three geodesics is still shortest. Hence in the end of this section, we will introduce another modified length for a graph connecting a triple of geodesics which will be called the Steiner length. The graph with the shortest Steiner length will be called the Steiner tree for the given triple of geodesics. The result of this section will be used in the proof of the main theorem.

\subsection{Triple of geodesics of \texorpdfstring{$\HHH$}{Lg} in general position}
In the previous section, we considered compact triangles in $\HH$. To avoid the degenerate cases, we adopt the convention that the three vertices are distinct and do not lie on the same geodesic. Similarly, we would also like to avoid certain degenerate cases when studying triples of geodesics. For this purpose, we say that three geodesics in $\HHH$ are in {\textit{general position}} if 
\begin{enumerate}
	\item they are pairwise disjoint;
	\item each pair does not share any end point on the boundary of $\HHH$;
	\item they are not all orthogonal to some geodesic in $\HHH$.
\end{enumerate}
Such a triple will be called a {\textit{generic triple}}. 

We can embed $\HH$ into $\HHH$ isometrically. Then we can consider the three vertices of a triangle in $\HH$ as the intersection points between $\HH$ and three geodesics in $\HHH$ orthogonal to it respectively. Hence the triangle case discussed in the previous section can be considered as a special case of this part.

\medskip

\noindent \textbf{Convention:} In the reminder of the paper, all triples of geodesics that we consider will be generic, and we will omit the word "generic". The indices in the discussion will be considered up to mod $3$. For consistency with the triangle case, we will use the same notation.

\medskip

Let $(\gamma_1,\gamma_2,\gamma_3)$ be a triple of geodesics in $\HHH$. Let $p$ be a point in $\HHH$. We denote by $p_1$, $p_2$ and $p_3$ its orthogonal projections to $\gamma_1$, $\gamma_2$ and $\gamma_3$ respectively. 
\begin{defn}
	The point $p$ is said to be {\textit{planar}} with respect to $(\gamma_1,\gamma_2,\gamma_3)$ if one of the following conditions is satisfied:
	\begin{enumerate}
		\item $p$ is not contained in any of the three geodesics, and it is contained in the hyperbolic plane determined by its projections $p_1$, $p_2$ and $p_3$;
		\item $p$ is on $\gamma_j$ for some $j$, such that $\angle p_{j-1}pu=\angle p_{j+1}pu'$, where $\gamma_j=[u,u']$.
 	\end{enumerate}
\end{defn}

\begin{rmk}
	If $p$ satisfies the first condition in the above definition, the plane containing it and its projections are unique. This is because that the plane contains the triangle determined by its three projections. If this triangle is degenerate, there will be some geodesic orthogonal to the three geodesics, contradicting to the fact that the three geodesics are in general position.
\end{rmk}	
\begin{prop}
	There is a unique planar point of type $(2)$ for each geodesic $\gamma_j$ in the triple. 
\end{prop}
\begin{proof}
	Let $\gamma_j=[u,u']$. For a point $p\in\gamma_j$, the angle $\angle p_{j+1}pu'$ is strictly increasing from $0$ to $\pi$ when $p$ moves from $u$ to $u'$, while at the same time, the angle $\angle p_{j-1}pu$ is strictly decreasing from $\pi$ to $0$. Hence there is a unique point $p$ where we have $\angle p_{j+1}pu'=\angle p_{j-1}pu$.
\end{proof}	
We denote the unique planar point on $\gamma_j$ by $v_j$. We denote by $v_{jj-1}$ and $v_{jj+1}$ its orthogonal projections to $\gamma_{j-1}$ and $\gamma_{j+1}$ respectively. We borrow the terminology from the triangle case and give the following definition:
\begin{defn}
	The angle $\angle v_{jj-1}v_jv_{jj+1}$ is called the {\textit{internal angle}} for $\gamma_j$.
\end{defn}
Similar to the triangle case, we define
\begin{defn}
	A planar point $p$ in $\HHH$ is said to be a {\textit{balanced point}} for $(\gamma_1,\gamma_2,\gamma_3)$ if $p$ is not contained in any of the three geodesics and it has the following property:
	\begin{equation}
		\angle p_1pp_2=\angle p_2pp_3=\angle p_1pp_3=\frac{2\pi}{3}.
	\end{equation}
\end{defn}
As before, we use $\widetilde{p}$ to denote a balanced point for a triple of geodesics.

\begin{defn}
	A triple $(\gamma_1,\gamma_2,\gamma_3)$ is said to be {\textit{$2\pi/3$-acute}} if all internal angles are strictly smaller than $2\pi/3$. If there is one internal angle greater or equal to $2\pi/3$, it is called {\textit{$2\pi/3$-obtuse}}.
\end{defn}

To a triple $(\gamma_1, \gamma_2, \gamma_3)$, we can associate  a unique right angled hexagon $H$ by considering the common perpendicular geodesics for the pairs of geodesics in the triple. Since we assume that the triple is in general position, this right-angled hexagon is compact as a subset of $\HHH$. 
\begin{defn}
	The right angled hexagon $H$ associated to a triple $(\gamma_1,\gamma_2,\gamma_3)$ is said to be {\textit{$2\pi/3$-acute}} (resp. {\textit{$2\pi/3$-obtuse}}) if the triple is $2\pi/3$-acute (resp. $2\pi/3$-obtuse).
\end{defn}
\begin{rmk}\label{specialcase}
	Notice that it is possible that the right angled hexagon is degenerate, because some of its sides on $\gamma_1$, $\gamma_2$ and $\gamma_3$ may have $0$ length. In particular, if all these three sides have $0$ length, we get to the triangle case. Then the definitions of being $2\pi/3$-acute and being $2\pi/3$-obtuse coincide with those that we defined in the previous section for a triangle.
\end{rmk}

\subsection{Fermat point for a triple of geodesics}
For any point $p\in\HHH$, we denote by $d_1(p)$, $d_2(p)$ and $d_3(p)$ the distances from $p$ to $\gamma_1$, $\gamma_2$ and $\gamma_3$ respectively. We define the function $L:\HHH\rightarrow \RR$ by 
\begin{equation}\label{lengthf}
L(p)=d_1(p)+d_2(p)+d_3(p).
\end{equation}
By the convexity lemma, we can see that the function $L$ for the triangle case is a restriction of the one defined by (\ref{lengthf}) to the plane containing the triangle. Similar to the triangle case, we give the following definition:
\begin{defn}
	The {\textit{Fermat point}} for $(\gamma_1,\gamma_2,\gamma_3)$ is the point in $\HHH$ which realizes the minimum of the function $L$. The graph obtained by connecting the Fermat point to its orthogonal projections to $\gamma_1$, $\gamma_2$ and $\gamma_3$ is called the {\textit{Fermat tree}} for $(\gamma_1,\gamma_2,\gamma_3)$.
\end{defn}
Similar to the triangle case, we have
\begin{thm}
	The Fermat point for $(\gamma_1,\gamma_2,\gamma_3)$ exists and is unique.
\end{thm}

\begin{proof}
	We use the extreme value theorem to prove this theorem.
		
	For a positive number $R\in\RR_{\ge0}$, we denote by $N_{j}(R)$ the $R$ neighborhood of $\gamma_j$. Since $\gamma_j$'s are disjoint in $\overline{\HHH}$, the set
	\begin{equation}
		N(R):=N_{1}(R)\cap N_{2}(R)\cap N_{3}(R)
	\end{equation}
	is compact for any $R$. Moreover, for $R<R'$, we have $N(R)\subset N(R')$.
	
	Let $R_H$ be the sum of the lengths of all sides of the right angled hexagon $H$ associated to $(\gamma_1,\gamma_2,\gamma_3)$. Then $N(R_H)$ is non-empty, since the points on $H\cap\gamma_1$ are contained in $N(R_H)$. Let $q$ be a point on $H\cap\gamma_1$, then we have $L(q)<R_H$. Therefore for any $p\in\HHH\setminus N(R_H)$, we always have $L(p)>R_H>L(q)$.
	
	Meanwhile, since $L$ is continuous on $\HHH$ and $N(R_H)$ is compact, we can use the extreme value theorem and get a point in $N(R_H)$ realizing the minimum of $L$ on $N(R_H)$. By the discussion above, this point also realizes the minimum of $L$ on $\HHH$. By definition, it is the Fermat point for $(\gamma_1,\gamma_2,\gamma_3)$.
	
	The uniqueness of the Fermat point follows from the convexity lemma. 
	
	Assume that there are two distinct points $p$ and $p'$ both realizing the minimum of $L$. Let $p_m$ denote the midpoint of the geodesic segment $pp'$. Then by the convexity lemma, for any geodesic $\gamma_j$ in the triple, we have:
	\begin{equation*}
		\mathrm{d}_j(p_m)\le \frac{1}{2}\mathrm{d}_j(p)+\frac{1}{2}\mathrm{d}_j(p'),
	\end{equation*}
	where  equality is attained when the geodesic passing through $p_1$ and $p_2$ is either $\gamma_j$ or orthogonal to $\gamma_j$. Therefore, we have 
	\begin{equation*}
		L(p_m)\le\frac{1}{2}L(p)+\frac{1}{2}L(p'),
	\end{equation*}
	where  equality holds if and only if for any $j=1,2,3$, the geodesic $\gamma_j$ is either the same as the geodesic passing through $p$ and $p'$, or orthogonal to it. Recall that as a convention, we only consider the triple $(\gamma_1,\gamma_2,\gamma_3)$ in general position. Hence, the inequality is strict, which contradicts the fact that $L(p_1)=L(p_2)$ is the minimal value of $L$ on $\HHH$.
\end{proof}
By considering the property of the Fermat point of a triangle, an immediate observation is:
\begin{ob}
	If $\widetilde{p}$ is the Fermat point of a triple $(\gamma_1,\gamma_2,\gamma_3)$, then it is planar and it is the Fermat point of the triangle $\Delta{\widetilde{p}_1\widetilde{p}_2\widetilde{p}_3}$. Therefore, if the Fermat point is not contained in any $\gamma_j$, it must be balanced. If it is on $\gamma_j$ for some $j$, then the angle $\angle\widetilde{p}_{j-1}\widetilde{p}\widetilde{p}_{j+1}$ is bigger than or equal to $2\pi/3$.
\end{ob}
The least obvious part of this observation is that when the Fermat point $\widetilde{p}$ is on $\gamma_j$ for some $j$, it must be planar. Without loss of generality, we may assume that $j=1$. The fact that $\widetilde{p}$ must be planar is a consequence of the following lemma:
\begin{lem}
	The unique planar point $v_1$ on $\gamma_1$ realizes the minimum of the sum of the distances to $\gamma_2$ and $\gamma_3$ among all points on $\gamma_1$.
\end{lem}
\begin{proof}
	Let $\gamma_1=[u,u']$. The existence of a point realizing the minimum of the sum of distances to $\gamma_2$ and $\gamma_3$ is clear, since the distance from $p\in\gamma_1$ to $\gamma_2$ and that from $p$ to $\gamma_3$ approaches infinity when $p$ goes to $u$ or $u'$.
	
	Consider a point $p\in\gamma_1$ and its projections $p_2$ and $p_3$ to $\gamma_2$ and $\gamma_3$ respectively. We consider the one parameter subgroup of $\PSLtwoC$ consisting of all elliptic isometries with axis $\gamma_1$. There is an isometry in this subgroup such that the image $p_3'$ of $p_3$ under this isometry is contained in the plane determined by $\gamma_1$ and $p_2$, moreover $p_3'$ and $p_2$ are separated by $\gamma_1$. Notice that $|pp_3|=|pp_3'|$.	
	\begin{center}
		\includegraphics[scale=0.5]{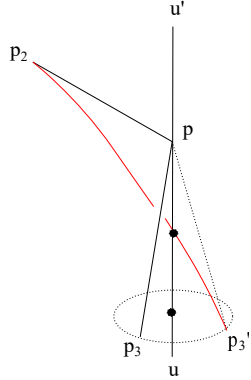}
		\captionof{figure}{The planar point realizes the minimum.}
	\end{center}
	If a point $p\in\gamma_1$ realizes the minimum of the sum of distance to $\gamma_2$ and $\gamma_3$ among all points in $\gamma_1$, then $p$ must realize the minimum of sum of distances to $p_2$ and $p_3$ among all points in $\gamma_1$, hence it realizes the minimum of the sum of the distance between $p_2$ and $p_3'$. Hence $p$ is contained in the geodesic segment $p_2p_3'$ and we have
	\begin{equation*}
		\angle p_2pu'=\angle p_3'pu=\angle p_3pu.
	\end{equation*}	
	Since $v_1$ is the only point with this property, we have $p=v_1$. 	
\end{proof}
Our next result is to show that the reciprocal of the observation is also true:
\begin{thm}\label{fpfortriple}
	Let $(\gamma_1, \gamma_2, \gamma_3)$ be a triple of geodesics in general position. Then,
	\begin{enumerate}
		\item if the triple admits a balanced point, then the balanced point is the Fermat point;
		\item if the triple is $2\pi/3$-obtuse, then the Fermat point of this triple is the vertex of the unique internal angle bigger or equal to $2\pi/3$.
	\end{enumerate}
\end{thm}
We first give two immediate corollaries that we can deduce from the above theorem:
\begin{cor}
	A triple admits a balanced point if and only if it is $2\pi/3$-acute, and the balanced point, if it exists, is unique.
\end{cor}
\begin{cor}
	If $(\gamma_1,\gamma_2,\gamma_3)$ is $2\pi/3$-obtuse, then there is a unique internal angle which is bigger or equal to $2\pi/3$. Hence, for any triple, the sum of interior angles is bounded from above by $7\pi/3$.
\end{cor}

The remainder of this subsection is occupied with the proof of the above theorem. We first give the proof for Fact (1) in the theorem: 
\begin{proof}[Proof of Fact (1) in Theorem \ref{fpfortriple}]
	Let $\widetilde{p}$ be a balanced point for $(\gamma_1,\gamma_2,\gamma_3)$. Let $\widetilde{p}_{j}$ denote the orthogonal projection of $\widetilde{p}$ to $\gamma_j$ for $j=1,2,3$. Let $P$ denote the plane containing the triangle $\Delta \widetilde{p}_{1}\widetilde{p}_{2}\widetilde{p}_{3}$. Let $P_j$ denote the plane orthogonal to $\widetilde{p}\widetilde{p}_{j}$ for $j=1,2,3$. Therefore, the geodesic $\gamma_j$ will be contained in $P_j$. We define the function $L':\HHH\rightarrow\RR_{\ge0}$ sending each point to the sum of its distances to the three planes $P_1$, $P_2$ and $P_3$. 
	
	Let $p$ be any point in $\HHH$. By definition, we have $L(p)\ge L'(p)$. Let $p'$ denote the image of $p$ under the reflection of $\HHH$ with respect to $P$. Then we have $L'(p)=L'(p')$. Let $p_m$ denote the mid point of $pp'$. Then $p_m$ is in $P$. Moreover, we have $L'(p_m)\le L'(p)$, where the equality holds if and only if $p$ is in $P$. In that case, the three points $p$, $p'$ and $p_m$ coincide.
	
	\begin{center}
		\includegraphics[scale=1.0]{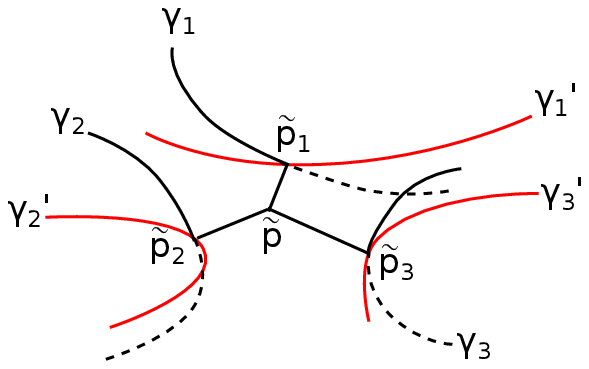}
		\includegraphics[scale=1.0]{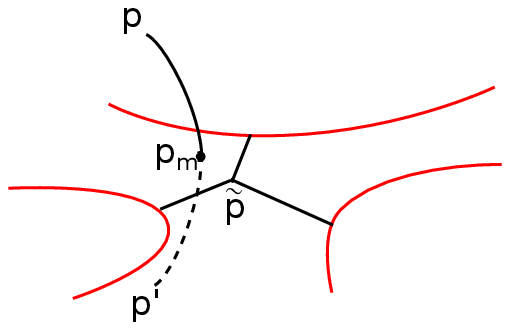}
		\captionof{figure}{The three red geodesics and $\widetilde{p}$ are coplane.}
	\end{center}
		
	Since $L'(\widetilde{p})=L(\widetilde{p})$, to show that $L(\widetilde{p})<L(p)$ for any $p\in\HHH$ different from $\widetilde{p}$, it is enough to prove the following fact:
	\begin{equation}\label{necessary}
		L'(\widetilde{p})< L'(p_m).
	\end{equation} 
	
	Let $\gamma_1'$, $\gamma_2'$ and $\gamma_3'$ be the intersection geodesics of $P$ with $P_1$, $P_2$ and $P_3$ respectively. Notice that they may have intersections points with each other. By their construction, the point $\widetilde{p}$ is not contained in any of these three geodesics. By the convexity lemma, $L'(p_m)$ equals to the sum of distances from $p_m$ to $\gamma_1'$, $\gamma_2'$ and $\gamma_3'$ respectively.

	By the above discussion, to prove  Fact (1), it is enough to prove that a balanced point $\widetilde{p}$ realizes the minimum of the sum of distances to $\gamma_1'$, $\gamma_2'$ and $\gamma_3'$ among all points in $P$.
	
	The complement of the three geodesics $\gamma_1'$, $\gamma_2'$ and $\gamma_3'$ in $P$ has several connected open components. We consider the one containing the balanced point $\widetilde{p}$ and denote it by $D$. Depending on the number of intersection points, the shape of $D$ has one of the following forms:
	
	\begin{center}
		\includegraphics[scale=0.7]{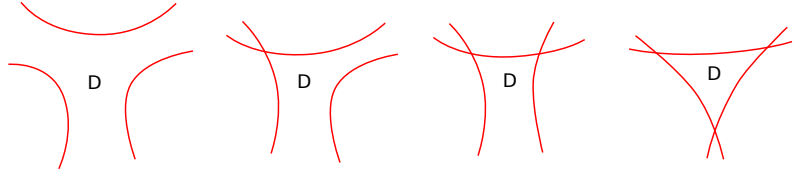}
		\captionof{figure}{Possible shapes of $D$}
	\end{center}
	
	By using a similar argument to the proof of Lemma \ref{out}, we can show that for any point $p\in P\setminus\overline{D}$, there is a point $q\in\overline{D}$, such that $L'(q)<L'(p)$. Therefore, from now on, we only need to consider points in $\overline{D}$.
	
	We consider the geodesic segments $\widetilde{p}\widetilde{p}_{1}$, $\widetilde{p}\widetilde{p}_{2}$ and $\widetilde{p}\widetilde{p}_{3}$. Their union separates $\overline{D}$ into three parts. Below we will prove that $L'(p)\ge L(\widetilde{p})$ for any $p$ contained in the part whose boundary contains $\widetilde{p}\widetilde{p}_{2}$ and $\widetilde{p}\widetilde{p}_{3}$ where equality holds if and only if $p=\widetilde{p}$. The proof for points in other two parts will be the same.
	
	\medskip
	
	Without loss of generality, we assume that $pp_1$ intersects $\widetilde{p}\widetilde{p}_{3}$ at $q$. Let $q_2$ denote the orthogonal projection of $q$ to $\gamma_2'$. 
	
	\begin{center}
		\includegraphics[scale=1.0]{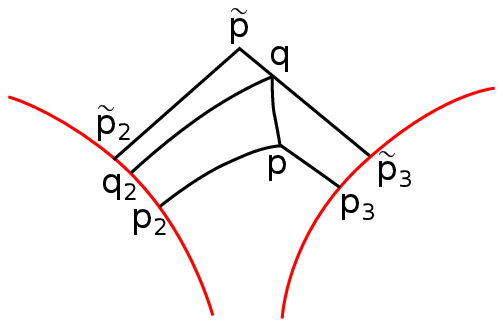}
		\captionof{figure}{}
	\end{center}

	Since $\angle\widetilde{p}_2\widetilde{p}\widetilde{p}_3=2\pi/3$, we have $\angle\widetilde{p}_3qq_2\ge2\pi/3$ where equality holds if and only if $q=\widetilde{p}$. We first prove the following inequality:
	\begin{equation}\label{1}
		|q\widetilde{p}_3|+|qq_2|\le |pp_2|+|pp_3|+|pq|,
	\end{equation}
	where equality holds if and only if $p=q$.
	
	If $p\neq q$, the geodesic segment $pq$ separates $\angle\widetilde{p}_3qq_2$ into two angles $\angle\widetilde{p}_3qp$ and $\angle pqq_2$. We denote them by $\alpha$ and $\beta$ respectively. Since $q\in\widetilde{p}\widetilde{p}_3$, we have $\alpha<\pi/3$. If $\beta\ge\pi/2$, then we have
	\begin{eqnarray*}
		|qq_2|&<&|pp_2|,\\
		|q\widetilde{p}_3|&\le&|pp_3|+|pq|.
	\end{eqnarray*}
    Hence,  the inequality in (\ref{1}) holds in this case.
    
    Now we assume that $\beta<\pi/2$. By using the cosine rule for a quadrilateral with two right angles, we have the following relations:
    \begin{eqnarray*}
    	\sinh|pp_3|&=&\sinh|q\widetilde{p}_3|\cosh|pq|-\cosh|q\widetilde{p}_3|\sinh|pq|\cos\alpha,\\
    	\sinh|pp_2|&=&\sinh|qq_2|\cosh|pq|-\cosh|qq_2|\sinh|pq|\cos\beta.    	
    \end{eqnarray*}
    We use the same technique as in the triangle case. We use $a$, $b$, $a'$, $b'$ and $c$ to denote $|q\widetilde{p}_3|$, $|qq_2|$, $|pp_3|$, $|pp_2|$ and $|pq|$ respectively. Let $t\in[0,1]$. We rewrite the above relations as follows:
    \begin{eqnarray*}
    	\sinh a'&=&\sinh a\cosh (tc+(1-t)c)-\cosh a\sinh(tc+(1-t)c)\cos\alpha,\\
    	\sinh b'&=&\sinh b\cosh(tc+(1-t)c)-\cosh b\sinh(tc+(1-t)c)\cos\beta.
    \end{eqnarray*}
    Then we have:
    \begin{eqnarray*}
    	\sinh a'&=&\sinh(a-tc)\cosh ((1-t)c)+\sinh(tc)\cosh(a+(1-t)c)(1-\cos\alpha)-\\
    	&&-\sinh((1-t)c)\cosh(a-tc)\cos\alpha,\\
    	\sinh b'&=&\sinh(b-(1-t)c)\cosh (tc)+\sinh((1-t)c)\cosh(b+tc)(1-\cos\beta)-\\
    	&&-\sinh(tc)\cosh(b-(1-t)c)\cos\beta.\\    	
    \end{eqnarray*}
	Our aim is to find $t\in[0,1]$, such that the two inequalities
	\begin{eqnarray*}
		\sinh a'&\ge&\sinh(a-tc)\cosh ((1-t)c),\\
		\sinh b'&\ge&\sinh(b-(1-t)c)\cosh (tc),
	\end{eqnarray*}
	hold. This would imply the inequality in (\ref{1})
		
	First, we notice that if $a+b\le c$, then we have the inequality in (\ref{1}) directly.
	
	Let us assume that $a+b>c$. If $a<c$ or $b<c$, we let $t=1$ or $t=0$ respectively and get the above two inequalities. Therefore we also have the inequality in (\ref{1}) in this case.
	
	Assume that $a>c$ and $b>c$. It is enough to find $t\in(0,1)$ such that the following inequalities hold:
	\begin{equation*}
		0<\frac{\cos\alpha}{1-\cos\alpha}\le\frac{\sinh(tc)}{\sinh((1-t)c)}\le\frac{1-\cos\beta}{\cos\beta}<\infty.
	\end{equation*}
	The same argument as in the proof of the $2\pi/3$-obtuse case in Theorem \ref{Fermat} implies the existence of such $t$. Therefore, the inequality in (\ref{1}) holds in this case. As a consequence, we have
	\begin{equation*}
		L'(p)\ge L'(q),
	\end{equation*}
	where the equality holds if and only if $p=q$.
	
	The remaining part of the proof is to compare $L'(q)$ with $L(\widetilde{p})$. The idea is similar to the proof of Lemma \ref{anglehyp2}. Let $q_1$ be the orthogonal projection of $q$ to $\gamma_1'$. Assume that $q\neq\widetilde{p}$. By the cosine rule for a quadrilateral with two right angles, we have the following relations:
	\begin{eqnarray*}
		\sinh|qq_2|&=&\sinh|\widetilde{p}\widetilde{p}_2|\cosh|\widetilde{p}q|+\frac{1}{2}\cosh|\widetilde{p}\widetilde{p}_2|\sinh|\widetilde{p}q|,\\
		\sinh|qq_1|&=&\sinh|\widetilde{p}\widetilde{p}_1|\cosh|\widetilde{p}q|+\frac{1}{2}\cosh|\widetilde{p}\widetilde{p}_1|\sinh|\widetilde{p}q|.
	\end{eqnarray*}
	They imply the following two equalities:
	\begin{eqnarray*}
		\sinh|qq_2|&=&\sinh(|\widetilde{p}\widetilde{p}_2|+\frac{|\widetilde{p}q|}{2})\cosh\frac{|\widetilde{p}q|}{2}+\sinh|\widetilde{p}\widetilde{p}_2|\sinh^2\frac{|\widetilde{p}q|}{2},\\
		\sinh|qq_1|&=&\sinh(|\widetilde{p}\widetilde{p}_1|+\frac{|\widetilde{p}q|}{2})\cosh\frac{|\widetilde{p}q|}{2}+\sinh|\widetilde{p}\widetilde{p}_1|\sinh^2\frac{|\widetilde{p}q|}{2},
	\end{eqnarray*}
	from which we conclude:
	\begin{eqnarray*}
		|qq_2|&>& |\widetilde{p}\widetilde{p}_2|+\frac{|\widetilde{p}q|}{2},\\
		|qq_1|&>& |\widetilde{p}\widetilde{p}_1|+\frac{|\widetilde{p}q|}{2}.
	\end{eqnarray*}
    
    Hence, we have $L'(q)>L(\widetilde{p})$ which moreover implies that $L(p)>L(\widetilde{p})$. Therefore the balanced point is the Fermat point for the triple $(\gamma_1,\gamma_2,\gamma_3)$.
\end{proof}

\medskip

To prove  Fact (2) in the theorem, we use similar ideas. The technical part is how to use the inequality (\ref{1}) that we proved above.
\begin{proof}[Proof of Fact (2) in Theorem \ref{fpfortriple}]
	Without loss of  generality, we assume that the internal angle at $v_1\in\gamma_1$ is greater or equal to $2\pi/3$, where $v_1$ is the unique planar point on $\gamma_1$. Let $v_{12}$ and $v_{13}$ denote its orthogonal projection to $\gamma_2$ and $\gamma_3$.
	
	If the internal angle at $v_1$ is $\pi$, then we are done. In that case, the three points $v_{12}$, $v_1$ and $v_{13}$ are collinear and the geodesic segment $v_{12}v_{13}$ is orthogonal to both $\gamma_2$ and $\gamma_3$. By the convexity lemma, the segment $v_{12}v_{13}$ realizes the distance between $\gamma_2$ and $\gamma_3$. Given any point $p\in\HHH$, when we connect it to its orthogonal projections to $\gamma_1$, $\gamma_2$ and $\gamma_3$ with geodesic segments, we always get a connected graph containing a subgraph connecting $\gamma_2$ to $\gamma_3$. Therefore the value $L(p)$ is always greater or equal to $L(v_1)$, and the equality is realized only when $p=v_1$. Hence $v_1$ is the Fermat point.
	
	From now on, we assume that the internal angle at $v_1$ is in $[2\pi/3,\pi)$. The first step is similar to the previous case. We would like to simplify the problem to the planar case. We consider the unique plane $P$ containing $v_1$, $v_{12}$ and $v_{13}$. Since $v_1$ is planar, there is a unique geodesic $\gamma(v_1)$ orthogonal to $\gamma_1$ with respect to which the geodesic containing $v_1v_{12}$ and that containing $v_1v_{13}$ are symmetric. It is contained in $P$ and separates the angle $\angle v_{12}v_1v_{13}$ equally. 
	
	We denote by $P_1$ the plane orthogonal to $\gamma(v_1)$ at $v_1$, by $P_2$ (resp. $P_3$) the plane orthogonal to $v_1v_{12}$ (resp. $v_1v_{13}$) at $v_2$ (resp. $v_3$). We denote by $\gamma_1'$, $\gamma_2'$ and $\gamma_3'$ the intersection of $P$ with $P_1$, $P_2$ and $P_3$ respectively. 
	
	\begin{center}
		\includegraphics[scale=1.0]{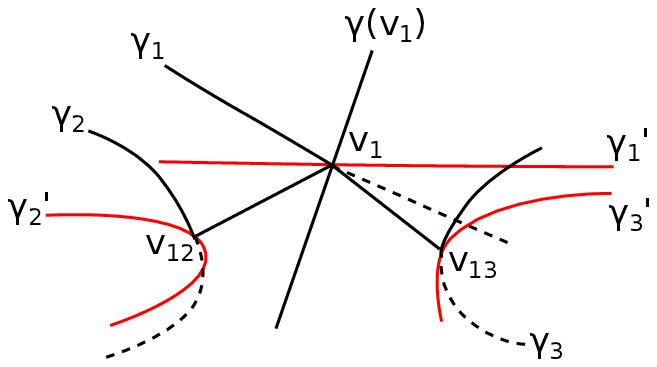}
		\captionof{figure}{}
	\end{center}
	
	Using a similar argument as in the previous case, to see $v_1$ is the Fermat point, it is enough to show that $v_1$ realizes the minimum of the sum of distances to $\gamma_1'$, $\gamma_2'$ and $\gamma_3'$ among all points in $P$. In particular, it is enough to consider the connected component of the complement of $\gamma_1'\cup\gamma_2'\cup\gamma_3'$ in $P$ containing $v_1v_{12}$ and $v_1v_{13}$. We denote the closure of this connected component by $D$. As in the previous case, we use $L'$ to denote the function sending each point $p$ to the sum of the distance from $p$ to $\gamma_1'$, $\gamma_2'$ and $\gamma_3'$.
	
	Given a point $p\in D$, we denote by $p_1$, $p_2$ and $p_3$ its orthogonal projections to $\gamma_1'$, $\gamma_2'$ and $\gamma_3'$. We consider the path of points $p$ such that $\angle p_{j-1}pp_{j+1}=2\pi/3$ and denote the path by $C_j$. Since $\angle v_{12}v_1v_{13}\ge 2\pi/3$, the intersection between $C_1$ and $D$ has two connected components. One intersects $\gamma_1'$ and $\gamma_2'$, the other intersects $\gamma_1'$ and $\gamma_3'$. We denote them by $C_{12}$ and $C_{13}$ respectively. The paths $C_2$ and $C_3$ are contained in $D$, such that $C_2$ intersects $\gamma_1'$ and $\gamma_3'$, while $C_3$ intersects $\gamma_1'$ and $\gamma_2'$. The subset $D_0$ of $D$ bounded by $C_3$, $C_{12}$, $\gamma_1'$ and $\gamma_2'$ is convex, as is the subset $D_0'$ bounded by $C_2$, $C_{13}$, $\gamma_1'$ and $\gamma_3'$. Moreover, the sets $D_0$ and $D_0'$ are disjoint, and the segment $pp_1$ for each point $p\in D$ intersects at most one of them. In particular, if $\angle v_{12}v_1v_{13}=2\pi/3$, the paths $C_{12}$, $C_{13}$, $C_2$ and $C_3$ meet at $v_1$.
	
	\begin{center}
		\includegraphics[scale=1.1]{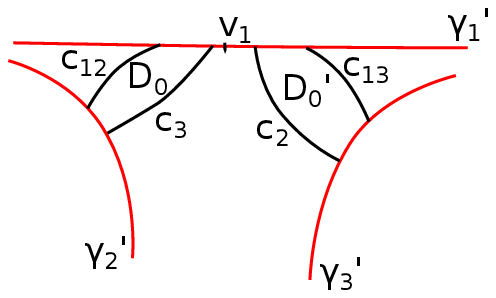}
			\captionof{figure}{}
	\end{center}
	
	We now claim that given a point $p$ in $D$, we can find a zig-zag path ending at $\gamma_1'$ starting from $p$ such that the $L'$-value is strictly decreasing along this path. More precisely, if $pp_j$ intersects with $D_0$, we denote by $q_1$ the point when $pp_j$ leaves $D_0$. By  inequality (\ref{1}), we have $L'(q_1)\le L'(p)$, where the equality holds if and only if $p=q_1$. Since $q$ is on the boundary of $D_0$, we can repeat the above process for $q_1$ and find another point $q_2$ on the boundary of $D_0$ such that $L'(q_2)<L(q_1)$. We keep repeating. If $\angle v_{12}v_1v_{13}>2\pi/3$, then the zig-zag path will stop on a point $q\in\gamma_1'$ within finite steps. Otherwise, we will have a zigzag path with infinite number of segments converging to $v_1$. 
	
	\begin{center}
		\includegraphics[scale=1.1]{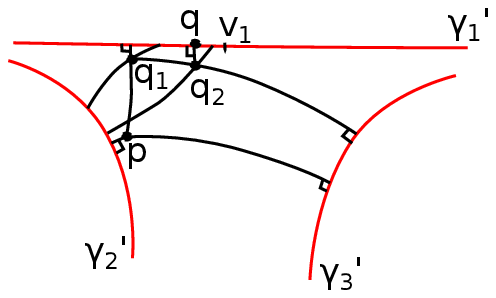}
			\captionof{figure}{}
	\end{center}
		
	As a conclusion, for any point $p\in D$ with $pp_j$ intersecting $D_0$ for some $j$, there is a point $q$ on $\gamma_1'$ whose $L'$-value is strictly smaller than $p$. The same argument works for $p$ with $pp_j$ intersecting $D_0'$ for some $j$. Notice that for any point $p$ in $D$, there is always some $j$ such that $pp_j$ intersecting $D_0$ or $D_0'$. Hence, we have the claim at the beginning of this paragraph.
	
	Since $v_1$ realizes the minimum of the $L'$-value among all points on $\gamma_1'$, Fact (2) holds.
	
\end{proof}

\subsection{The Steiner tree for triples of geodesics in \texorpdfstring{$\HHH$}{Lg}}
Given a triple $(\gamma_1, \gamma_2,\gamma_3)$, we consider the corresponding triple of involutions $(r_1,r_2,r_3)$. In this part, we assume that $G$ is a finite graph connecting the three geodesics in the triple, whose edges are all geodesics segments. Hence the intersection between an edge of $G$ and $\gamma_j$ for any $j$ is either a point or the entire edge. For our convenience, if the intersection is a point, we add this point to the set of vertices of $G$. If edges in $G$ intersect, the intersection points are also added in the vertex set.

We are interested in the length of the graph $D$ which arises from doubling $G$ using $r_j$ for some $j$, and then taking the quotient by the action of the group  $\langle r_1r_2, r_2r_3 \rangle$. Note that $D$ is a graph of rank at least two. The length of $D$ will be the length of a carrier graph of the associated representation. Moreover, generically the length of $D$ is twice of the Steiner length of $G$ which is defined as follows:

\begin{defn}
	The {\textit{Steiner length}} $\mathsf{L}(G)$ of a graph $G$ is defined by:
	\begin{equation}
	\mathsf{L}(G)=L(G_1)+\frac{1}{2}L(G_2),
	\end{equation}
	where $G_2$ is the intersection $G\cap (\gamma_1\cup \gamma_2\cup \gamma_3)$ and $G_1$ is the complement of $G_2$ in $G$.
\end{defn}

\begin{rmk}
	It is possible that $G$ has two distinct edges $e$ and $e'$ collinear, adjacent to a same vertex $v$ lying in $\gamma_j$ for some $j$, while $e$ and $e'$ are both orthogonal to $\gamma_j$. In that case, the intersection $G\cap r_j(G)$ contains $(e\cup e')\cap r_j(e\cup e')$ which is not contained in $\gamma_j$. In this case, the double of the Steiner length of $G$ is strictly bigger than the usual length of $D$, since the Steiner length for $(e\cup e')\cap r_j(e\cup e')$ equals to the usual length and is counted twice when doubling. 
	
	However, this is not an issue for our main purpose, since we are interested in the length of the double of the graph. We can always replace $G$ by a different graph $G'$ without changing the double graph $D$, such that the above situation does not happen to $G'$. 
	
	More precisely, since $G$ is connected, without loss of generality, we may assume that $e$ can be connected to $\gamma_{j-1}$ by a subgraph $F_1$, without passing $e'$. Now there are two possible cases. The first case is that $e$ can also be connected to $\gamma_{j+1}$ without passing through $e'$ by a subgraph $F_2$. In this case, we can cut out $F_1\cup F_2$ from $G$ and denote by $F'$ the connected component containing $e'$. Then we denote by $G'$ the union $(G\setminus F')\cup r_j(F')$. Notice that $G'$ is connected and finite. It connects $\gamma_j$'s. Moreover, its double is still $D$, but the double of its Steiner length now equals the usual length of $D$.
	
	The second case is that there is no such a subgraph $F_2$ described above. Then by cutting $G$ at $v$, we have two connected component. We denote by $F_1'$ the one containing $F_1$. We consider $G''$ the union $F_1'\cup r_j(G\setminus F_1')$. Notice that $G''$ is connected, finite. It connects $\gamma_{j-1}$, $\gamma_j$ and $r_j(\gamma_{j+1})$. Notice that its double is also $D$ and the double of its Steiner length equals to the usual length of $D$.
\end{rmk}

Consider the Fermat point for $(\gamma_1,\gamma_2,\gamma_3)$. Recall that the Fermat graph is obtained by connecting the Fermat point to its orthogonal projections to the geodesics in the triple. In particular, the $\mathsf{L}$-value equals to the usual length of the Fermat graph. By the property of the Fermat point, we see that this graph has the smallest usual length among all graphs connecting $\gamma_i$'s. Hence it is a good candidate for our purpose. However, its Steiner length is not always the shortest. Below is an example where this is the case. 

\begin{ex}
	We consider the case where $\gamma_1$, $\gamma_2$ and $\gamma_3$ are lying in the same plane and disjoint from each other in $\HHH\cup\partial\HHH$. We further assume that none of them separates the other two into two different half planes. Assume that the Fermat point $p$ of $(\gamma_1,\gamma_2,\gamma_3)$ is on $\gamma_1$. Let $p_2$ and $p_3$ be its orthogonal projections to $\gamma_2$ and $\gamma_3$ respectively. By our discussion in the previous section, the angle $\angle p_2pp_3$ is greater than or equal to $2\pi/3$. Suppose that it is greater. Let $G_p$ denote the Fermat tree of $(\gamma_1,\gamma_2,\gamma_3)$. 
	
	We construct a new graph $G'$ from $G$ in the following way. Choose points $q_1$ and $q_2$ on $\gamma_1$ such that the angles $\angle p_2q_2q_3$ and $\angle p_3q_3q_2$ are both $2\pi/3$. It is easy to see that $p$ is contained in the segment $q_2q_3$. The new graph $G'$ is the union of the geodesic segments $p_2q_2$, $q_2q_3$ and $q_3p_3$.
	\bigskip
	
	\begin{center}
		\includegraphics[scale=0.6]{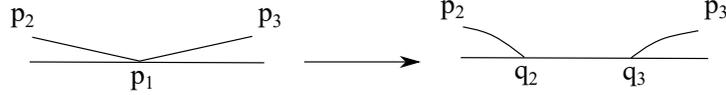}
		\captionof{figure}{Breaking the Fermat point}
	\end{center}
	By Lemma \ref{anglehyp2}, we know that:
	\begin{eqnarray*}
		|q_2p_2|+\frac{1}{2}|q_2p|<|pp_2|.\\
		|q_3p_3|+\frac{1}{2}|q_3p|<|pp_3|.
	\end{eqnarray*}
	Therefore, the new graph has shorter Steiner length. Note that this particular graph may not have the shortest Steiner length since $p_i$ is not necessarily the projection of $q_i$ to $\gamma_i$ for $i=2,3$.
\end{ex}

\begin{defn}
	The {\textit{Steiner tree}} of $(\gamma_1,\gamma_2,\gamma_3)$ is a connected graph $G$ connecting them with the shortest Steiner length $\mathsf{L}(G)$.
\end{defn}

Our next result is about the existence of the Steiner tree for a triple of geodesics and its characterization.
\begin{prop}
	For any triple $(\gamma_1,\gamma_2,\gamma_3)$, the Steiner tree exists. Moreover it is one of the following two types:
		\begin{enumerate}[(a)]
			\item It is $G_p$;
			\item It is a piece-wise geodesic with three segments where the middle one is contained in one of $\gamma_i$'s and the other two segments intersects $\gamma_{i-1}$ and $\gamma_{i+1}$ orthogonally. Moreover, the angle between two adjacent segments is $2\pi/3$.
		\end{enumerate}
\end{prop}

\begin{proof}
	In this proof, by length we will always mean the Steiner length. Recall that we only consider finite connected graphs. All graphs that we will consider always connect the three geodesics $\gamma_j$'s.
	
	Since removing one edge to break a loop always reduces the length of the graph, it is enough to only consider simply connected graphs. In the rest of the proof, we assume that this is always the case.
	
	We first prove the following lemma: 
	\begin{lem}\label{123}
		Given any finite simply connected graph $G$, its length is always bigger or equal to that of a graph with one of the following three types:
		\begin{enumerate}
			\item One valence $3$ vertex and three valence $1$ vertices;
			\item One valence $2$ vertex and two valence $1$ vertices;
			\item Two valence $2$ vertices and two valence $1$ vertices.
		\end{enumerate}
	\end{lem}
	\begin{proof}[{\it Proof of Lemma \ref{123}}]
		Consider the path $P_{12}$ (resp. $P_{13}$) with shortest Steiner length in $G$ connecting $\gamma_1$ and $\gamma_2$ (resp. $\gamma_1$ and $\gamma_3$). Notice that $P_{12}$ (resp. $P_{13}$) is a piecewise geodesic and meet each one of $\gamma_1$ and $\gamma_2$ (resp. $\gamma_1$ and $\gamma_3$) once. Let $p_2$ and $p_3$ denote the intersections $P_{12}\cap \gamma_1$ and $P_{13}\cap \gamma_1$ respectively.  
		
		Notice that the choices for $P_{12}$ and $P_{13}$ may not be unique. We choose and fix one choice. In the following, we discuss case by case.
		
		\noindent \textbf{Case 1}: $P_{12}$ intersects $\gamma_3$. We consider the orientation of $P_{12}$ from $\gamma_1$ to $\gamma_2$ and denote by $p_1$ the point where $P_{12}$ meets $\gamma_3$ for the first time and by $p_2$ the point where $P_{12}$ leaves $\gamma_3$ for the last time. Let $q_1$ and $q_2$ be the projection of $p_1$ and $p_2$ to $\gamma_1$ and $\gamma_2$ respectively. Then the graph $G'$ given by the union of segments $q_1p_1$, $p_1p_2$ and $p_2q_2$ is of type $(2)$ or $(3)$ with shorter length than $G$. 
		\smallskip
		
		\noindent \textbf{Case 2}: $P_{13}$ intersects $\gamma_2$. The argument is the same as in Case 1.
		\smallskip
		
		\noindent \textbf{Case 3}: both $P_{12}\cap\gamma_3$ and $P_{13}\cap\gamma_2$ are empty and $P_{12}\cap P_{13}\neq\emptyset$. We consider the point on $P_{12}$ where it starts from $\gamma_2$ and meets $P_{13}$ for the first time. Then by connecting this point to its projections to $\gamma_j$'s respectively, we obtain a graph of type $(1)$ with length shorter than $G$.
		\smallskip
		
		\noindent \textbf{Case 4}: both $P_{12}\cap\gamma_3$ and $P_{13}\cap\gamma_2$ are empty and $P_{12}\cap P_{13}=\emptyset$. We denote by $Q_{23}$ the shortest path in $G$ connecting $P_{12}$ and $P_{13}$. Hence, the path $Q_{23}$ does not contain any edges in $P_{12}$ or in $P_{13}$, and only meet each one of them once. Denote by $o_2$ and $o_3$ the vertices $P_{12}\cap Q_{23}$ and $P_{13}\cap Q_{23}$ respectively. All vertices in the subgraph $Q_{23}$ except $o_2$ and $o_3$ have valence $2$.
		\smallskip
		
		\noindent \textbf{Case 4.1}: $Q_{23}$ intersects $\gamma_1$, $\gamma_2$ and $\gamma_3$. Since $Q_{23}$ is a piecewise geodesic, it contains a subgraph which starts from $\gamma_i$, intersects $\gamma_j$ for some times, then ends at $\gamma_k$, where $\{i,j,k\}=\{1,2,3\}$. Using the same argument as for case 1, we can find a path of type $(2)$ or $(3)$ with shorter length than $G$.
		\smallskip
		
		\noindent \textbf{Case 4.2}: $Q_{23}$ is disjoint from $\gamma_2$, but intersects $\gamma_1$ or $\gamma_3$. Then we consider $Q_{23}'$ the union of $Q_{23}$ and the subgraph of $P_{12}$ between $\gamma_2$ and $o_2$. By our construction, the graph $Q_{23}'$ is a simple piecewise geodesic connecting $\gamma_1$, $\gamma_2$ and $\gamma_3$, hence it contains a subgraph starting from $\gamma_2$, intersecting $\gamma_j$ form some times, then ending at $\gamma_k$ with $\{j,k\}=\{1,3\}$. Using the same argument as for case 1, this we can find a path of type $(2)$ or $(3)$ with shorter length than $G$.
		
		\smallskip
		
		\noindent \textbf{Case 4.3}: $Q_{23}$ is disjoint from $\gamma_3$, but intersects $\gamma_1$ or $\gamma_2$. We connect $Q_{23}$ to $\gamma_3$ using the subgraph of $P_{13}$ between $o_3$ and $\gamma_3$, then the proof is the same as in Case 4.2.
		\smallskip
		
		\noindent \textbf{Case 4.4}: $Q_{23}$ is disjoint from $\gamma_1$, but intersects $\gamma_2$ or $\gamma_3$. We connect $Q_{23}$ to $\gamma_1$ using the subgraph of $P_{12}$ between $o_2$ and $\gamma_1$, then the proof is the same as in Case 4.2.
		\smallskip
		
		\noindent \textbf{Case 4.5}: $Q_{23}$ is disjoint from $\gamma_1$ and $\gamma_2$, but intersects $\gamma_3$. We connect $P_{12}$ to $\gamma_3$ using the subgraph of $Q_{23}$ between $o_2$ and its first intersection point with $\gamma_3$ counting from $o_2$. Then we consider the graph given by connecting $o_2$ to its three projections to $\gamma_1$, $\gamma_2$ and $\gamma_3$. It is of type $(1)$ with length shorter than that of $G$.
		\smallskip
		
		\noindent \textbf{Case 4.6}: $Q_{23}$ is disjoint from $\gamma_1$ and $\gamma_3$, but intersects $\gamma_2$.	We connect $P_{13}$ to $\gamma_2$ using the subgraph of $Q_{23}$ between $o_3$ and its first intersection point with $\gamma_2$ counting from $o_3$, then the proof is the same as in Case 4.5.
		\smallskip
		
		\noindent \textbf{Case 4.7}: $Q_{23}$ is disjoint from $\gamma_2$ and $\gamma_3$, but intersects $\gamma_1$. We connect $Q_{23}$ to $\gamma_2$ and $\gamma_3$ using parts of $P_{12}$ and $P_{13}$, such that the resulting graph meets each of $\gamma_2$ and $\gamma_3$ once, and has only valence $2$ vertices except those two on $\gamma_2$ and $\gamma_3$. Then the proof is the same as in Case 1.
	\end{proof}
	Returning to the proof of the proposition, if $G$ is of type $(1)$ or type $(2)$, we can moreover assume that $G$ has no edge contained in any of the $\gamma_j$'s. Otherwise, either the edge on $\gamma_j$ is removable, or there are two of $\gamma_j$'s intersect each other, which contradict to the fact that $\gamma_j$'s are in general position. Therefore, its Steiner length is greater than or equal to that of the Fermat tree. 
	
	If $G$ is of type $(3)$, we may moreover assume that the middle edge is on $\gamma_j$ for some $j$. Otherwise, its length is the usual length, hence greater than that of the Fermat tree.  Without loss of generality, we may assume that the graph is the union of $p_2q_2$, $q_2q_3$ and $q_3p_3$, where $q_2q_3$ is contained in $\gamma_1$, and $p_2$ and $p_3$ are in $\gamma_2$ and $\gamma_3$ respectively. We may further assume  that $p_2$ (resp. $p_3$) is the orthogonal projection of $q_2$ (resp. $q_3$), otherwise, we can shorten the Steiner length by replacing $p_2$ (resp. $p_3$) by the orthogonal projection. Then $G$ is determined by the positions of $q_2$ and $q_3$ and the Steiner length is a function of $q_2$ and $q_3$. By the extreme value theorem, there is a graph $G$ which realizes the minimal Steiner length among all type $(3)$ graphs with middle edge on $\gamma_1$. Moreover it is either of type $(2)$ or of type $(3)$. When it is of type $(2)$, its length is greater than that of the Fermat tree. When it is of type $(3)$, by Lemma \ref{anglehyp2} and the proof of Case I in Theorem \ref{Fermat}, the angle between each pair of adjacent edges must be $2\pi/3$.
\end{proof}
\begin{rmk}\label{twosteinertree}
	Unlike the uniqueness result of the Fermat point (or Fermat graph), it is possible to have up to two Steiner trees for a triple $(\gamma_1,\gamma_2,\gamma_3)$. An example would be a triple of coplanar geodesics. Consider a triple $(\gamma_1,\gamma_2,\gamma_3)$ contained in the same plane, such that no one separates other two. We assume that the distances between pairs of them are $(a,a,2a)$ with $a>0$. By studying the planar points for this triple of geodesics, we can see that there exists a balanced point for this triple of geodesics. By our previous discussion, this balanced point is the Fermat point $p$, and the Fermat tree $G$ is given by connecting the $p$ to $p_j$'s which are its projections to $\gamma_j$'s respectively.
	
	Assume that the distance between $\gamma_1$ and $\gamma_2$ is $2a$. Consider the involution $r_3$ with respect to $\gamma_3$. We denote by $\gamma_1'$ and $\gamma_2'$ the images of $\gamma_1$ and $\gamma_2$ under $r_3$ respectively. We denote by $D$ the double of $G$ under $r_3$. An observation is that the subgroup generated by the elliptic element $\eta$ of $\HHH$ fixing $p_3$, preserving the plane with angle $\pi/2$ preserves the set $\{\gamma_1,\gamma_2,\gamma_1',\gamma_2'\}$. Hence the image $\eta(D)$ is still a graph connecting $\gamma_1$, $\gamma_2$, $\gamma_1'$ and $\gamma_2'$. Moreover the middle edge of $\eta(D)$ lying on $\gamma_3$. Therefore, the graph $\eta(D)$ is the double of a graph $G'$ connecting $(\gamma_1,\gamma_2,\gamma_3)$ with one edge on $\gamma_3$.
	
	By our assumption, we can see that $L(G)=L(G')$. Both of them are Steiner tree for $(\gamma_1,\gamma_2,\gamma_3)$, while $G$ is of type $(a)$ and $G'$ is of type $(b)$
	\begin{center}
		\includegraphics[scale=0.5]{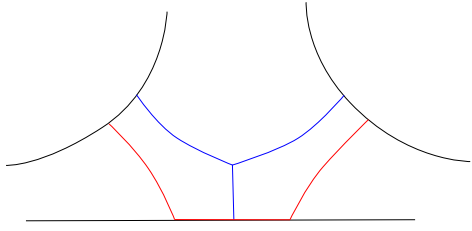}
		\captionof{figure}{Two Steiner trees}
	\end{center}
\end{rmk}
\subsection{The fourth edge orientation on \texorpdfstring{$\Sigma$}{Lg}}
Recall that $\Sigma$ is the tree of superbases. In the earlier section, we have discussed three edge orientation on $\Sigma$. In this part, we would like to associated to $\Sigma$ the fourth edge orientation using the Steiner tree introduced above.

More precisely, let $\rho$ be an irreducible representation of $\FF_2$ into $\SLtwoC$ sending all primitive elements to loxodromic elements. Given an ordered triple $(X,Y,Z)$, by proposition \ref{irre1}, there is a triple of $\pi$-rotation $(r_X,r_Y,r_Z)$ such that:
\begin{eqnarray*}
	\rho(X)&=&r_Yr_Z,\\
	\rho(Y)&=&r_Zr_X,\\
	\rho(Z)&=&-r_Xr_Y.
\end{eqnarray*}
We denote by $\gamma_X$, $\gamma_Y$ and $\gamma_Z$ the axes of $r_X$, $r_Y$ and $r_Z$ respectively. Since $\rho(X)$, $\rho(Y)$ and $\rho(Z)$ are all loxodromic elements by assumption, the triple $(\gamma_X,\gamma_Y,\gamma_Z)$ is generic.

Although the triple of geodesics depends on the choice of representative of a superbasis, the relative position among the three geodesics in the triple does not. Given two representatives of a superbasis, the two associated triples of geodesics are different by a diagonal action of an isometry of $\HHH$. In this way, we associate to a superbasis a triple of geodesics up to isometry. By our work in the earlier part of this section, it admits a Steiner tree. 

Recall that $\widetilde{\Omega}$ consists of pairs $([W],v)$ where $[W]$ is the $\Inn^{\mathfrak{e}}$-orbit of a primitive element $W$ and $v$ is a vertex on the boundary of connected component in $\DD\setminus\Sigma$ identified with $[W]$. We can define a function $L_\rho:\widetilde{\Omega}\rightarrow\RR_{>0}$ sending $([W],v)$ to the Steiner length of the Steiner tree associated to $v$. This function $L_\rho$ then induces an edge orientation on $\Sigma$, such that the orientation goes from vertex with bigger $L_\rho$-value to the one with smaller $L_\rho$-value.
\begin{prop}\label{attractingtree4}
	If $\rho$ satisfies BQ-conditions, then there is a compact attracting subtree for the edge orientation induced by $L_\rho$.
\end{prop}
\begin{proof}
	Consider the right angled hexagon $H(X,Y,Z)$ associated to $(\gamma_X,\gamma_Y,\gamma_Z)$. Without loss of generality, we may assume that with respect to the fixed basis $(X_0,Y_0)$, the element $Z$ has bigger word length than $X$ and $Y$. Let $e$ denote the edge of $\Sigma$ adjacent to the components $[X_0]$ and $[Y_0]$.
	
	In the proof of proposition \ref{attractingsubtree3}, we saw that if the vertex $(X,Y,Z)$ is at distance $N$ to the edge $e$ with $N$ large enough, the complex lengths of the three sides of $H(X,Y,Z)$ on $\gamma_X$, $\gamma_Y$ and $\gamma_Z$ will be uniformly close to $i\pi$, $i\pi$ and $0$ respectively. It does not depend on the choice of the geodesic on $\Sigma$.
	
	The internal angle on $\gamma_Z$ is bounded from below by the transport angle along $\gamma_Z$, and the planar point on $\gamma_Z$ is contained in the side of $H(X,Y,Z)$ on $\gamma_Z$. Therefore, there exists a positive integer $N$, such that if the distance from $(X,Y,Z)$ to the starting vertex is bigger than $N$, the associated right angled hexagon is $2\pi/3$-obtuse. Moreover, the associated Steiner tree is of type (2), meaning that it has one valence $2$ vertex on $\gamma_Z$ and two valence $1$ vertices on $\gamma_X$ and $\gamma_Y$ respectively. The Steiner length is more or less the sum of the lengths of two sides of $H(X,Y,Z)$ on the axes of $\rho(X)$ and $\rho(Y)$. The difference is bounded from above by the sum of the lengths of the sides of $H(X,Y,Z)$ on $\gamma_X$, $\gamma_Y$ and $\gamma_Z$.
	
	Consider going one step farther than $(X,Y,Z)$ and denote the vertex by $(X',Y,Z)$. Then the right angled hexagon $H(X',Y,Z)$ is also $2\pi/3$-obtuse and its Steiner tree $G(X',Y,Z)$ is of type (2). Up to isometry, the triple of geodesics for $(X',Y,Z)$ is $(\gamma_X,\gamma_Y,r_X(\gamma_Z))$. The valence $2$ vertex $v_{X'}$ of $G(X',Y,Z)$ is on $\gamma_X$. The internal angle on $\gamma_X$ change from $2\pi/3$-acute to $2\pi/3$-obtuse. Hence the Steiner length of $G(X,Y,Z)$ is strictly smaller than that of $G(X',Y,Z)$. Therefore, there is a compact attracting subtree of $\Sigma$ with respect to the edge orientation induces by $L_\rho$.
	\begin{center}
		\includegraphics[scale=0.6]{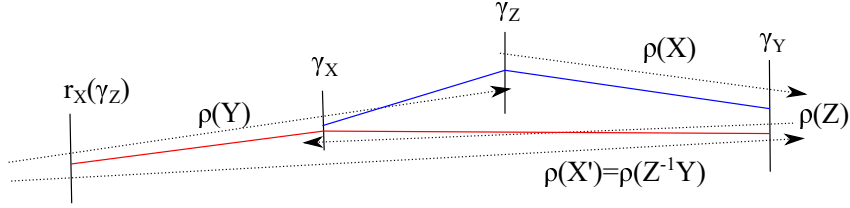}
		\captionof{figure}{More one step further}
	\end{center}
\end{proof}
\section{The carrier \texorpdfstring{$2$}{Lg}-graphs for an irreducible representation}
Recall that the definition of a carrier $n$-graph $\Gamma$ for a hyperbolic $3$-manifold $M$ involves an immersion of $\Gamma$ into $M$. Let $\rho$ be a lift of the representation of $\pi_1(M)$ into $\SLtwoC$ corresponding to the hyperbolic structure on $M$. By taking the lift to the universal cover, we obtain an immersion of the universal cover $\widetilde{\Gamma}$ of $\Gamma$ into $\HHH$ which is $\rho$-equivariant. This immersion of $\widetilde{\Gamma}$ is unique up to the action of $\rho(\pi_1 (M))$. Following this idea, we generalize the definition of a carrier graph to that for an irreducible representation. Recall that we are interested in irreducible representations $\rho:\FF_2\rightarrow\SLtwoC$ and we will only consider $2$-graphs.
\begin{rmk}
	One may define  carrier graphs using any graph with no restriction on valence, however by lemma \ref{anglehyp2}, the carrier graph with minimal length, if exists, is always a trivalent graph. As such, we only will consider trivalent $2$-graphs.
\end{rmk}

\subsection{Definition}
Let $\Gamma$ be a trivalent $2$-graph. Then there are only two possible combinatorial types of $\Gamma$ which we call the Buckle and the Dumbbell, see figure below.

\begin{center}
	\includegraphics[scale=0.8]{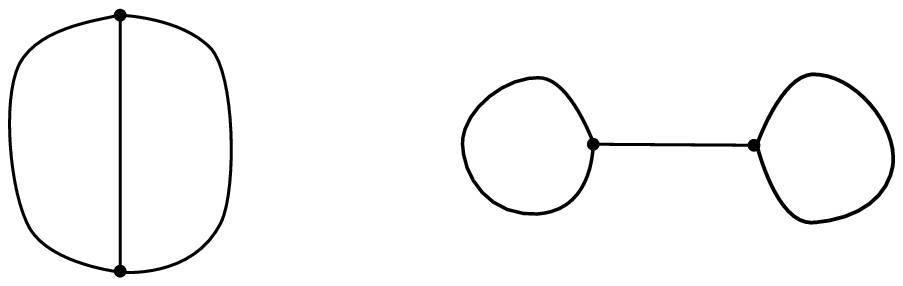}
	\captionof{figure}{A Buckle and a Dumbbell}
\end{center}

We can give $\Gamma$ a marking by assigning a free basis $(X,Y)$ to two distinct simple loops on $\Gamma$ respectively. We call it a marked $2$-graph. For a marked $2$-graph $(\Gamma, (X,Y))$, a fundamental domain $D$ in $\widetilde{\Gamma}$ can be chosen to be the one with four boundary points $p_1$, $p_1'$, $p_2$ and $p_2'$, such that $X(p_1)=p_1'$ and $Y(p_2)=p_2'$. As a normalization, when $\Gamma$ is equipped with a metric $l$, the points $p_1$ and $p_2$ will be chosen so that their projections on $\Gamma$ are the two mid-points of two edges of $\Gamma$ where they lie on respectively. We will only consider this fundamental domain $D$ for each marked $2$-graph. Depending on the marked $2$-graph, the fundamental domain $D$ has $2$ possible combinatorial types given in the figure below:
\begin{center}
	\includegraphics[scale=0.5]{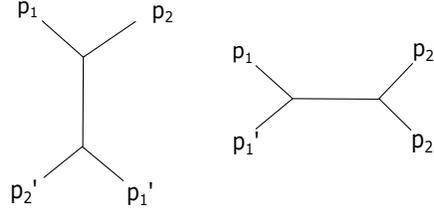}
	\captionof{figure}{Possible types of $D$}
\end{center}

\begin{defn}
	Two marked $2$-graphs $(\Gamma_1,(X_1,Y_1))$ and $(\Gamma_2,(X_2,Y_2))$ are said to be {\textit{equivalent}}, if the following holds:
	\begin{enumerate}
		\item There exists a homeomorphism from $\Gamma_1$ to $\Gamma_2$;
		\item This homeomorphism induces an automorphism of $\FF_2$ sending $(X_1,Y_1)$ to $(X_2,Y_2)$. 
	\end{enumerate}
\end{defn}
If we consider the space of all marked $2$-graphs and take its quotient by the above equivalence relation,  we obtain $2$ classes corresponding to the above two types of fundamental domain respectively. To avoid talking about class and taking representative, we will choose and fix one representative in each class with the same marking $(X_0,Y_0)$. These two graphs will be the topological model for the later discussion. As a convention, in the reminder of this paper, by a $2$-graph, we mean one of the two topological models and simply denote it by $\Gamma$.

\begin{rmk}
	A marking on a graph induces an isomorphism between the fundamental group of the graph and $\FF_2$. By choosing and fixing a marking $(X_0,Y_0)$, we choose and fix the isomorphism.
\end{rmk}

\medskip

Now let $\rho$ be an irreducible representation of $\FF_2$ into $\SLtwoC$. 
\begin{defn}
	A local homeomorphism $\Psi:\widetilde{\Gamma}\rightarrow\HHH$ is said to be $\rho$-{\textit{equivariant}} if there exists a free basis $(X,Y)$ of $\FF_2$ such that the following commutative diagram holds:
		\begin{equation*}
			\xymatrixcolsep{5pc}\xymatrix{
				\widetilde{\Gamma}\ar[d]^{\Psi}\ar[r]^{W}&\widetilde{\Gamma}\ar[d]^{\Psi}\\
				\HHH\ar[r]^{\rho(h_\Psi(W))}&\HHH}
		\end{equation*}
		where $h_\Psi$ is the automorphism of $\FF_2$ sending $X_0$ and $Y_0$ to $X$ and $Y$ respectively.
\end{defn}
\begin{defn}
	A {\textit{marked carrier}} $2$-{\textit{graph}} for $\rho$ is a pair $(\Gamma,\Psi)$, where
	\begin{itemize}
		\item $\Gamma$ is a $2$-graph;
		\item $\Psi:\widetilde\Gamma\rightarrow\HHH$ is a $\rho$-equivariant local homeomorphism.
	\end{itemize}
\end{defn}
The pullback of the intrinsic metric on $\Psi(\widetilde{\Gamma})$ by $\Psi$ induces a metric on $\Gamma$. We denote by $l(\Gamma,\Psi)$ this metric on $\Gamma$. 
\begin{defn}
	The {\textit{length}} of $\Gamma$ with respect to $l(\Gamma,\Psi)$ is defined to be the sum of the $l(\Gamma,\Psi)$-length of all the edges. We denote it by $L(\Gamma,\Psi)$.
\end{defn}

\begin{defn}
	Two marked carrier 2-graphs $(\Gamma_1,\Psi_1)$ and $(\Gamma_2,\Psi_2)$ are said to be {\textit{equivalent}} if 
	\begin{enumerate}
		\item $\Gamma_1$ and $\Gamma_2$ are the same $2$-graph;
		\item there exists an element $W$ of $\FF_2$, such that $Im(\Psi_1)=\rho(W)(Im(\Psi_2))$;
		\item the two automorphisms $h_{\Psi_2}$ and $h_{\Psi_1}$ satisfies:
			\begin{equation*}
				h_{\Psi_2}=h'h_{\Psi_1},
			\end{equation*}
		      where $h'\in\Inn^\mathfrak{e}\FF_2$.
	\end{enumerate}
    An equivalence class of marked carrier graph is called a {\textit{carrier}} 2-{\textit{graph}} for $\rho$. We denote it by $(\Gamma,[\Psi])$. Its {\textit{length}} $L(\Gamma,[\Psi])$ is defined to be the length $L(\Gamma,\Psi)$ for a representative $(\Gamma,\Psi)$ in the equivalence class.
\end{defn}
\begin{rmk}
	This definition is similar to the one for the equivalence class of marked hyperbolic structure on $3$-manifold. More precisely, the conditions (1) says that the two $2$-graphs homeomorphic to each other, the condition (2) implies that the two $2$-graphs equipped with induced metrics are isometric and the condition (3) says that the marking on the two $2$-graphs are the same. For $3$-manifold, the condition $(2)$ is equivalent to the two pullback metric are isometric. This is not the case for carrier graph. Instead, the condition (2) is strictly stronger than that two pullback metrics are isometric, since we consider the intrinsic metrics.
\end{rmk}

\subsection{Decomposition of the space of carrier \texorpdfstring{$2$}{Lg}-graphs}
Since we are interested in looking for the carrier $2$-graph with shortest length, it is reasonable to only consider the geodesic carrier $2$-graphs, meaning that those carrier $2$-graphs $(\Gamma, [\Psi])$, such that all edges of $\widetilde{\Gamma}$ are sent to geodesics segments by $\Psi$. By lemma \ref{anglehyp2}, we can moreover assume that the edges of $\widetilde{\Gamma}$ are sent to geodesic segments  with strictly positive length.

Let $\mathcal{CG}_\rho$ denote the space of all geodesic carrier $2$-graphs associated to $\rho$. By our definition, the carrier graphs can be classified using $h_\Psi$. Let $(X,Y)$ be a basis of $\FF_2$. We denote by $\mathcal{CG}_\rho([X],[Y])$ the subspace consisting of equivalence classes of all marked carrier $2$-graphs with $h_\Psi(X_0)\in[X]$ and $h_\Psi(Y_0)\in [Y]$. Therefore we can write $\mathcal{CG}_\rho$ into a disjoint union of $\mathcal{CG}_\rho([X],[Y])$'s.

\begin{defn}
	A carrier $2$-graph $(\Gamma,[\Psi])\in\mathcal{CG}_\rho$ is said to be {\textit{critical}}, if it realizes the minimum of the length function restricted on $\mathcal{CG}_\rho([X],[Y])$ for some $(X,Y)$. It is called {\textit{minimal}} if it has the minimal length among all carrier graphs.
\end{defn}

Now let us recall the main result:
\begin{mainthm}
	If $\rho:\FF_2\rightarrow\SLtwoC$ satisfies the BQ-conditions, then it admits finitely many critical carrier $2$-graphs.
\end{mainthm}
As a corollary, we have
\begin{maincor}
	If $\rho:\FF_2\rightarrow\SLtwoC$ satisfies the BQ-conditions, then it admits finitely many minimal carrier $2$-graphs.
\end{maincor}

\section{Proof of the main result}
Let $\rho$ be a BQ-representation. Let $(X,Y)$ be a basis of $\FF_2$. We first consider the carrier $2$-graphs in $\mathcal{CG}_\rho([X],[Y])$. To simplify the notation, we will also use $D$, $p_1$, $p_1'$, $p_2$, $p_2'$, $v_1$ and $v_2$ to denote their $\Psi$-images. We consider the two ordered triple $(X,Y,Z)$ and $(X,Y^{-1},Z')$.

Consider the triple $\SLtwoC$-elements $(\rho(X),\rho(Y),\rho(Z))$. By Proposition \ref{irre1}, we can find a triple of $\pi$-rotations $(r_X,r_Y,r_Z)$ with axes $(\gamma_X,\gamma_Y,\gamma_Z)$ in general position, such that:
\begin{eqnarray*}
	\rho(X)&=&r_Yr_Z,\\
	\rho(Y)&=&r_Zr_X,\\
	\rho(Z)&=&-r_Xr_Y,\\
	\rho(Z')&=&(r_Zr_Xr_Z)r_Y.
\end{eqnarray*}
We first prove the following lemma:
\begin{lem}
	If $r_Z(D)\neq D$, then there is a new carrier graph $(\Gamma',\Psi')$ with its length smaller than that of $(\Gamma,\Psi)$.
\end{lem}
\begin{proof}
	Assume that $r_Z(D)\neq D$. We first consider the case where $\Gamma$ is a buckle. We consider the geodesic segments $p_1r_Z(p_1')$, $p_1'r_Z(p_1)$, $p_2r_Z(p_2')$, $p_2'r_Z(p_2)$, $v_1r_Z(v_2)$ and $v_2r_Z(v_1)$. Below is a picture for the planar case.
	
	\begin{center}
		\includegraphics[scale=0.8]{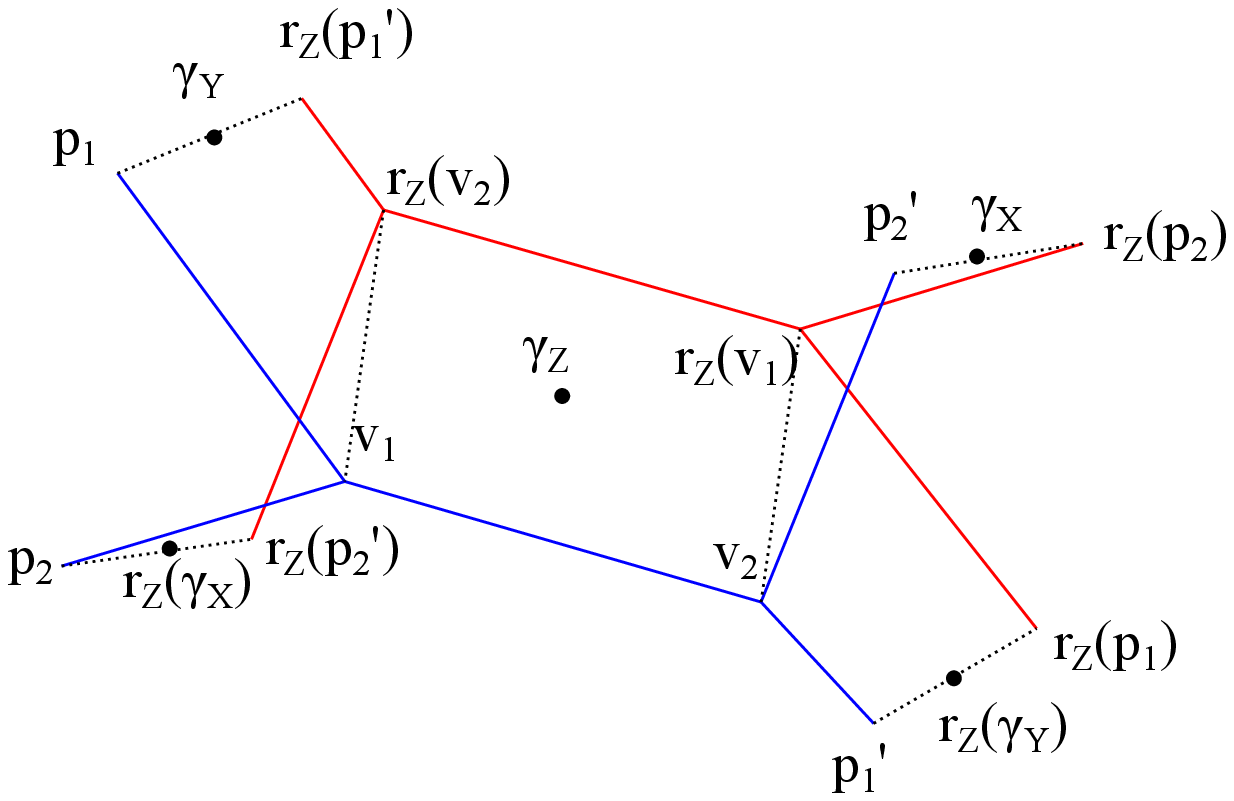}
		\captionof{figure}{}
	\end{center}	

	Then we can find a new graph $D'$ homotopic to $D$ relative to these $6$ geodesic segments, such that its vertices are the mid points of the six geodesic segments respectively. The hyperbolic convexity lemma implies that the length of $D'$ is smaller than that of $D$. Therefore the new carrier graph $(\Gamma',\Psi')$ induced by $D'$ has smaller length than $(\Gamma,\Psi)$.
		
	If $\Gamma$ is a dumbbell, then we consider the geodesic segments $v_1r_Z(v_1)$ and $v_2r_Z(v_2)$ instead of $v_1r_Z(v_2)$ and $v_2r_Z(v_1)$. The rest of the proof is the same as the first case.
\end{proof}
Now assume that $r_Z(D)=D$. Then the graph $D$ can be viewed as branched $2$-cover of a graph $G$ in $\HHH$ connecting $(\gamma_X,\gamma_Y,\gamma_Z)$ or $(r_z(\gamma_X),\gamma_Y,\gamma_Z)$ with covering map induced by $r_Z$, such that the branching part is either $v_1v_2$ or the midpoint of $v_1v_2$. 

We consider $G(X,Y,Z)$ (resp. $G(X,Y^{-1},Z')$) the Steiner tree for $(\gamma_X,\gamma_Y,\gamma_Z)$ (resp. $(r_z(\gamma_X),\gamma_Y,\gamma_Z)$) and denote by $D(X,Y,Z)$ (resp. $D(X,Y^{-1},Z')$) its double. By the definition of the Steiner tree, we can see that if $(\Gamma,\Psi)$ is a marked carrier $2$-graph with $(h_{\Psi}(X_0),h_{\Psi}(Y_0))=(X,Y)$ or $(X,Y^{-1})$, then its length is strictly longer than that of $D(X,Y,Z)$ or that of $D(X,Y^{-1},Z')$. Therefore, up to isometry, the carrier graph realizing the infimum of the length among all carrier $2$-graph in $\mathcal{CG}_\rho([X],[Y])$ is induced by the shorter one of $D(X,Y,Z)$ and $D(X,Y^{-1},Z')$. Therefore to look for the minimal graph, it is enough to consider those induced by the Steiner trees associated to superbases. 

Given an edge orientation on the tree of superbasis, we call a vertex is a sink if all three adjacent edges are oriented towards it. By the proposition \ref{attractingtree4}, we know that there exists a compact attracting subtree $\Sigma_0$ with respect to the edge orientation induced by $L_\rho$. This moreover implies that there are finitely many sinks.
\begin{lem}
	If a vertex is a sink, then the corresponding Steiner tree induces a critical carrier graph.
\end{lem}
\begin{proof}
	It is enough to show that if the Steiner tree does not induces a critical carrier graph, the Steiner tree of a neighbor vertex has strictly shorter length.
	
	Consider the superbasis $(X,Y,Z)$. If the corresponding carrier graph $(\Gamma,\Psi)$ is not critical, it has a valence $4$ vertex. Then the Steiner tree for $(\gamma_X,\gamma_Y,\gamma_Z)$ consists of one vertex and two edges. Assume that the vertex is on the geodesic $\gamma_Z$. We denote by $v$ this vertex and $v_X$ and $v_Y$ its projections to $\gamma_X$ and $\gamma_Y$ respectively. By our previous discussion, the angle between $vv_X$ and $vv_Y$ is bigger or equal to $2\pi/3$, and the angle between each edge and the $\gamma_Z$ is bigger or equals to $\pi/3$. This implies that the angle between $vv_X$ and $vr_Z(v_Y)$ is smaller than $2\pi/3$. Hence it is not the Steiner tree for the triple $(\gamma_X,r_Z(\gamma_Y),\gamma_Z)$. Therefore the new triple of geodesics admits a shorter Steiner tree, which contradicts to the facts that $(X,Y,Z)$ is a sink.
\end{proof}
This shows the existence of critical carrier graphs. The finiteness is given by the fact that the Steiner graph associated to a vertex outside the compact attracting subtree consists of one vertex and two edges. Therefore we prove the main theorem.

The corollary follows from the fact that a minimal carrier graph is also a critical carrier graph.

\section{Example}
Consider $\HH$ as a totally geodesic submanifold in $\HHH$. In this section we will consider the BQ-representations which preserve $\HH$ in $\HHH$. They all come from the hyperbolic structure on the surfaces with Euler characteristic $-1$. Based on the topological type of the corresponding surface, these representations can be classified into the following $4$ types:
\begin{enumerate}
	\item once-punctured torus;
	\item three-punctured sphere;
	\item once-punctured M\"obius band;
	\item once-punctured Klein bottle.
\end{enumerate}
For more details, one may read \cite{goldman2009ems} and \cite{goldmanmcshanestantchevtan2016}.

A minimal carrier $2$-graph for these BQ-representations is induced by a superbasis $(X,Y,Z)$ of which the triple of geodesics $(\gamma_X,\gamma_Y,\gamma_Z)$ are of special configuration, and the number of minimal carrier $2$-graph is at most $2$. We will give the description of these configuration case by case.

We recall three types of a Steiner tree for a triple of geodesics:
\begin{enumerate}
	\item One valence $3$ vertex and three valence $1$ vertices;
	\item One valence $2$ vertex and two valence $1$ vertices;
	\item Two valence $2$ vertices and two valence $1$ vertices.
\end{enumerate}

\subsection{Once-punctured torus}
Let $\rho$ be an irreducible representation. Given a superbasis $(X,Y,Z)$, we assume that the three geodesics $\gamma_X$ , $\gamma_Y$ and $\gamma_Z$ associated to $(\rho(X),\rho(Y),\rho(Z))$ are all orthogonal to $\HH$. Such a representation is a BQ representation if and only if it corresponds to a hyperbolic structure on once-punctured torus. If the boundary holonomy is hyperbolic, the hyperbolic structure is complete and of infinite volume; if the boundary holonomy is parabolic, the hyperbolic structure is complete and of finite volume; if the boundary holonomy is elliptic, then the hyperbolic structure is not complete and the boundary becomes a conic singularity. In the first two cases, the representation is discrete and faithful, while in the last case, it is not discrete and possibly not faithful, either.

For each superbasis, the corresponding right angled hexagon is a triangle and the corresponding Steiner tree is either of type (1) or of type (2). Therefore the Steiner tree is the same as the Fermat tree in this case. We start with the triangle $\Delta v_1v_2v_3$ for one superbasis. To get the triangle for one neighbor superbasis, we consider the involution of $\HH$ with respect to $v_j$ for some $j$, and consider its action on a different vertex $v_{j+1}$ to get a new point. The new triangle is determined by $v_j$, $r_j(v_{j+1})$ and $v_{j-1}$. It is clear that the two triangle have same lengths on two sides adjacent to $v_j$ and the internal angles at $v_j$ are complementary to each other.
\medskip
\begin{center}
	\includegraphics[scale=0.5]{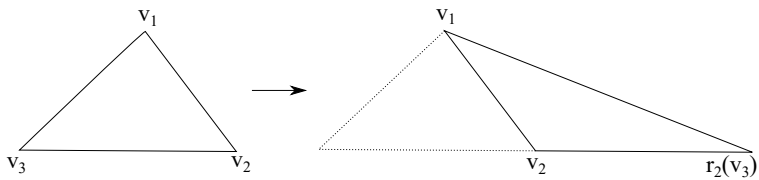}
	\captionof{figure}{Changing superbases}
\end{center}

In this way, we can obtain one triangle for each superbases inductively. We can verify the following two facts:
\begin{enumerate}[a.]
	\item there are finitely many superbases of which the associated triangles are $2\pi/3$-acute;
	\item if there is a basis $(X,Y)$ of $\FF_2$, such that the axes of $\rho(X)$ and $\rho(Y)$ are orthogonal, then there are only two superbases $(X,Y,Z)$ and $(X,Y^{-1},Z')$ of which the associated triangles are non obtuse in the usual sense, and are right angled triangles; otherwise there is a unique superbasis of which the associated triangle is non obtuse and is acute.
\end{enumerate} 

The Fermat trees for those $2\pi/3$-acute triangles induce critical carrier $2$-graphs, while the Fermat tree for the acute triangle or the right-angled triangle induces the minimal carrier $2$-graph.

\subsection{three-holed sphere}
Let $\rho$ be an irreducible representation. Given a superbasis $(X,Y,Z)$, we assume that the three geodesics $\gamma_X$ , $\gamma_Y$ and $\gamma_Z$ associated to $(\rho(X),\rho(Y),\rho(Z))$ are contained in $\HH$ and disjoint from one another. Such a representation is a BQ representation if and only if it corresponds to a hyperbolic structure on thrice-punctured sphere. Since the three peripheral elements are all primitive elements, they must be sent to hyperbolic elements in $\SLtwoR$ by $\rho$. Therefore, the hyperbolic structure on the three punctured sphere is complete and with infinite volume. The representation is always discrete and faithful.

The triple of geodesics $(\gamma_X,\gamma_Y,\gamma_Z)$ has one of the following two configurations:
\begin{enumerate}[a.]
	\item there exists one of the three geodesics separating the other two into two different half planes;
	\item no one separate the other two into two half planes.
\end{enumerate}
\begin{center}
	\includegraphics[scale=0.5]{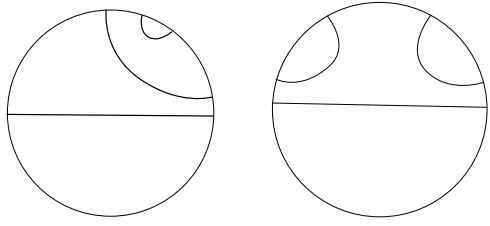}
	\captionof{figure}{Possible configurations}
\end{center}
There is a unique superbasis of which the associated triple of geodesics are of the second configuration. It admits
\begin{itemize}
	\item either a unique Steiner tree of type (1) or type (3);
	\item or two Steiner tree, such that one is of type (1), while the other is of type (3).
\end{itemize}
For the former case, the carrier $2$-graph induced by the Steiner tree is the unique minimal one, and for the latter case, the two carrier $2$-graph induced by the two Steiner trees are both minimal (see Remark \ref{twosteinertree}).

\subsection{Once-punctured M\"obius band}
The once-punctured M\"obius band and the once-punctured Kleinian bottle are both non-oriented surfaces. Their corresponding representations have been studied by Goldman-McShane-Stantchev-Tan in \cite{goldmanmcshanestantchevtan2016}.

Consider a BQ-representation $\rho$ corresponding to a hyperbolic structure on a once-punctured M\"obius band. Given a superbasis $(X,Y,Z)$, the three geodesics $\gamma_X$, $\gamma_Y$ and $\gamma_Z$ associated to $(\rho(X),\rho(Y),\rho(Z))$ have one orthogonal to $\HH$ and two others contained in $\HH$. Without loss of generality, we may assume that $\gamma_X$ is the one orthogonal to $\HH$ and we denote by $v_X$ its intersection point with $\HH$. Since the triple of geodesics are in general position, there are two possible configurations:
\begin{enumerate}[a.]
	\item one of $\gamma_Y$ and $\gamma_Z$ separates the other two into two half planes in $\HH$;
	\item none of $\gamma_Y$ and $\gamma_Z$ separate the other two.
\end{enumerate}
\begin{center}
	\includegraphics[scale=0.5]{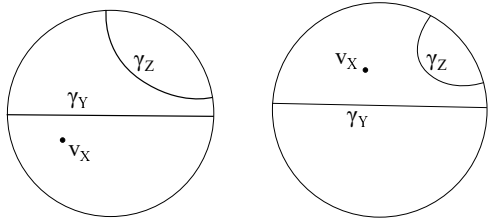}
	\captionof{figure}{Possible configurations}
\end{center}
Let $r_X$ denote the involution of with respect to $\gamma_X$. For a triple of second configuration, by applying $r_X$ on $\gamma_Y$ or $\gamma_Z$, the new triple is almost always of the first configuration. The exceptions are two triples associated to two superbases which are neighbors to each other. They are the special superbases among all superbases for the representation $\rho$. The two triples of geodesics associated to these two superbases are both of the second configuration and their internal angles at $v_X$ are complementary to each other.

If the two internal angles at $v_X$ are both right angles, then the Steiner tree associated to these two triples induce the only two minimal carrier $2$-graphs for $\rho$; if one is bigger than the other, then the Steiner tree for the triple with smaller internal angle at $v_X$ induces the unique minimal carrier $2$-graph of $\rho$. 
\begin{center}
	\includegraphics[scale=0.5]{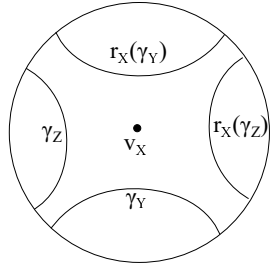}
	\captionof{figure}{Realizing minimal carrier graph}
\end{center}

\subsection{Once-punctured Klein bottle}
Consider a BQ-representation $\rho$ corresponding to a hyperbolic structure on a once-punctured Klein bottle. Given a superbasis $(X,Y,Z)$, the triple of geodesics $(\gamma_X,\gamma_Y,\gamma_Z)$ associated to $(\rho(X),\rho(Y),\rho(Z))$  has two orthogonal to $\HH$, and the last one is contained in $\HH$. Without loss of generality, we may assume that $\gamma_X$ is the one contained in $\HH$ and we denote by $v_Y$ and $v_Z$ the intersection points between $\gamma_Y$ and $\gamma_Z$ with $\HH$ respectively. Since the triple of geodesics are in general position, there are two possible configurations:
\begin{enumerate}[a.]
	\item the geodesic $\gamma_X$ separates $v_Y$ from $v_Z$ in $\HH$;
	\item the two point $v_Y$ and $v_Z$ are on the same side of $\gamma_X$ in $\HH$.
\end{enumerate}
\begin{center}
	\includegraphics[scale=0.5]{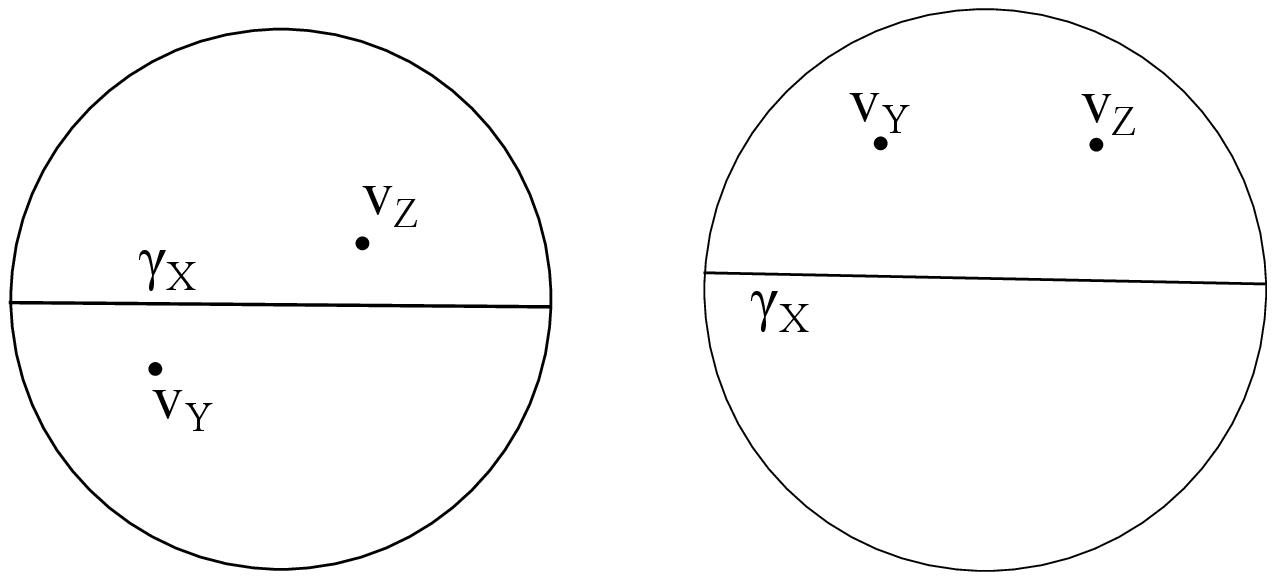}
	\captionof{figure}{Possible configurations}
\end{center}
The geodesic in $\HH$ passing $v_Y$ and $v_Z$ is the axis of $\rho(X)$. Denote it by $\delta_X$. The triple induces the minimal carrier $2$-graph is the one such that
\begin{itemize}
	\item the geodesic $\delta_X$ is disjoint from $\gamma_X$;
	\item the distance between $\delta_X$ and $\gamma_X$ is closest among all disjoint pairs $(\delta_{X'},\gamma_{X'})$'s;
	\item either the two point $v_Y$ and $v_Z$ are separated by the common perpendicular geodesic $\gamma$ between $\delta_X$ and $\gamma_X$, or one of them is on $\gamma$.
\end{itemize}
If $v_Y$ and $v_Z$ are separated by $\gamma$, then there is a unique minimal carrier $2$-graph; if one of the two points is on $\gamma$ and assume that it is $v_Y$, then there are two minimal carrier $2$-graph induced by $(\gamma_X,\gamma_Y,\gamma_Z)$ and $(\gamma_X,\gamma_Y,r_Y(\gamma_Z))$.
\begin{center}
	\includegraphics[scale=0.5]{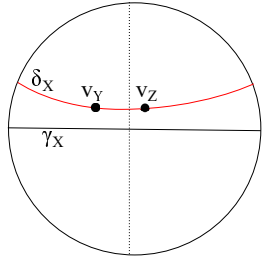}
	\captionof{figure}{Realizing minimal carrier graph}
\end{center}

\appendix
\section{Proof of the convexity lemma}
\begin{lem1}[Hyperbolic Convexity Lemma]
	Let $a$, $b$, $c$ and $d$ be four points in $n$-dimensional hyperbolic space $\mathbb{H}^n$. Let $e$ and $f$ be two points on $ab$ and $cd$ respectively, such that $|ae|/|ab|=|df|/|cd|=t$ for some $t\in(0,1)$. Then,
	\begin{equation}
	|ef|\le (1-t)|ad|+t|bc|.
	\end{equation}
	Moreover the equality is realized if and only if all four points lie on the same geodesic and $m$ is between $e$ and $f$.
\end{lem1}
\begin{proof}
	Let us first consider the triangle $\Delta abc$ with vertices $a$, $b$ and $c$. There is always a $2$ dimensional hyperbolic plane containing this triangle. Let $m$ be a point on $ac$ such that $|am|/|ac|=t$, where $0< t<1$. Then $|em|/|bc|\le t$. the proof is as follows.
	
	We will use the Poincar\'e disk model $\DD$ of $\HH$. By translations and rotations of $\DD$, we can assume that the vertex $a$ is at the origin and $b$ is point on the positive side of the real axis. Recall that the hyperbolic metric is given by the following formula under polar coordinates:
	\begin{equation}
	\mathrm{d}\,s(r,\theta)=\frac{2\sqrt{(\mathrm{d}\,r)^2+r^2(\mathrm{d}\,\theta)^2}}{(1-r^2)}.
	\end{equation}
	Then the point on the geodesic segment $bc$ has the coordinates $(r_1(\theta),\theta)$ with $\theta\in[0,\alpha]$ where $\alpha=\angle bac$. Let us assume first that $0<\alpha<\pi$. Then we consider the path $\gamma$ connecting points $e$ and $m$ which is parametrized by $(t\cdot r_1(\theta),\theta)$.		
	\begin{center}
		\includegraphics[scale=0.5]{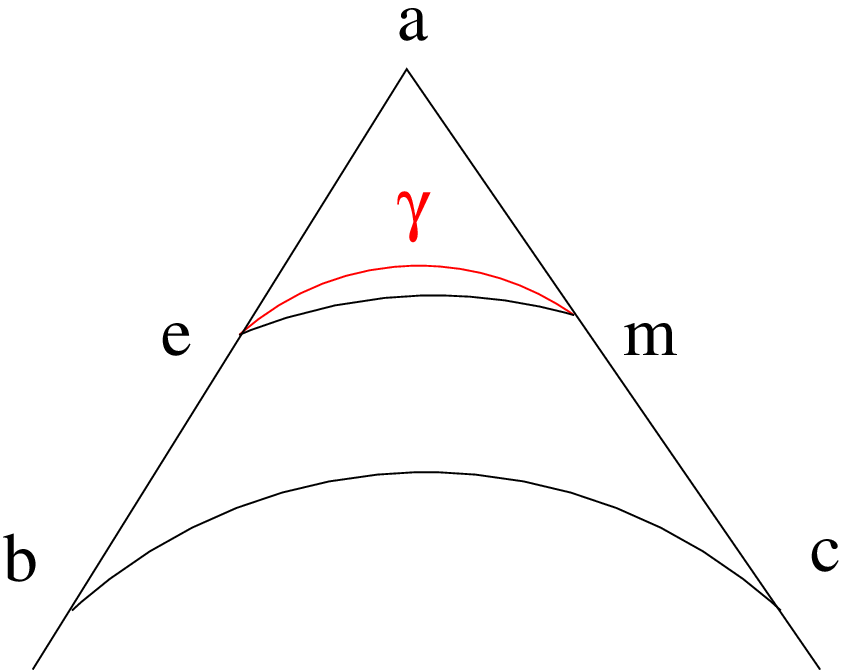}
		\captionof{figure}{}
	\end{center}
	
	The length of $\gamma$ can be given by the following two integrals:
	\begin{eqnarray*}
	|\gamma|&=&\int\limits_{0}^{\alpha}\frac{2t\sqrt{(\frac{\mathrm{d}\,r_1}{\mathrm{d}\,\theta})^2+r_1^2}\mathrm{d}\,\theta}{(1-t^2r_1^2)}\\
	&=&t\int\limits_{0}^{\alpha}\frac{2\sqrt{(\frac{\mathrm{d}\,r_1}{\mathrm{d}\,\theta})^2+r_1^2}\mathrm{d}\,\theta}{(1-t^2r_1^2)}\\
	&<&t\int\limits_{0}^{\alpha}\frac{2\sqrt{(\frac{\mathrm{d}\,r_1}{\mathrm{d}\,\theta})^2+r_1^2}\mathrm{d}\,\theta}{(1-r_1^2)}\\
	&=&t|bc|.
	\end{eqnarray*}
	Since $em$ is the geodesic connecting $e$ and $m$, its length is strictly shorter than that of $\gamma$. This shows that $|em|/|bc|< t$ when $0<\alpha<\pi$.
	
	When $\alpha=0$ or $\pi$, then we have the equality. Combining the two cases, we show that $|em|/|bc|\le t$. By applying the same discussion to the triangle $\Delta acd$, we can show that $|mf|/|ad|\le (1-t)$.
	\begin{center}
		\includegraphics[scale=0.5]{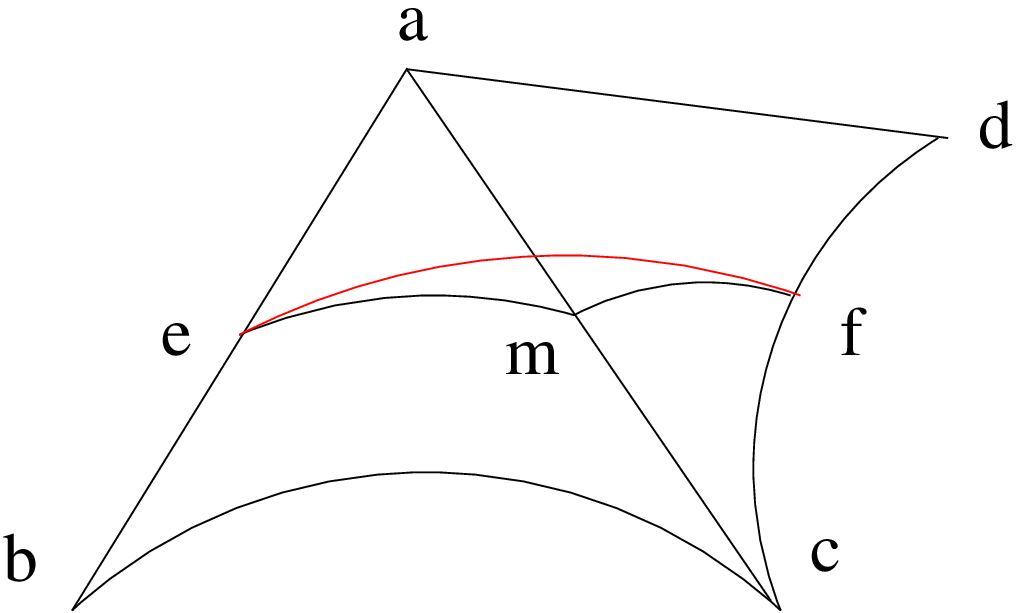}
		\captionof{figure}{}
	\end{center}
	
	By the triangular inequality, we can conclude the following inequality:
	\begin{equation}
	|ef|\le|em|+|fm|\le t|bc|+(1-t)|ad|
	\end{equation}
	
	The equality is realized if and only if
	\begin{eqnarray*}
		|em|&=&t|bc|\\
		|fm|&=&(1-t)|ad|\\
		|ef|&=&|em|+|fm|
	\end{eqnarray*}
	The first equality is realized if and only if $a$, $b$  and $c$ are collinear. The second equality is realized if and only if $a$, $c$ and $d$ are collinear. The third one is realized if and only if $e$, $m$ and $f$ are on the same geodesic and $m$ is between $e$ and $f$. This complete the proof of this lemma.
\end{proof}
\begin{rmk}
	To be more precise,  the equality $|ef|=t|bc|+(1-t)|ad|$ is realized if and only if $a$, $b$, $c$ and $d$ are collinear and in one of the following orders: $\overrightarrow{abdc}$, $\overrightarrow{adbc}$, $\overrightarrow{abcd}$, $\overrightarrow{badc}$, $\overrightarrow{bacd}$, $\overrightarrow{bcad}$, $\overrightarrow{dacb}$, $\overrightarrow{dcab}$, $\overrightarrow{cdab}$, $\overrightarrow{dcba}$, $\overrightarrow{cbda}$ and $\overrightarrow{cdba}$.
\end{rmk}

\end{document}